\documentclass[11pt,letterpaper]{amsart}
\usepackage{geometry}
\usepackage[authoryear,sort&compress]{natbib}
\bibpunct{[}{]}{;}{n}{,}{,}

\usepackage{amsmath}
\usepackage{amsfonts}
\usepackage{amssymb}
\usepackage{amsaddr}
\usepackage{booktabs}
\usepackage{comment}
\usepackage[shortlabels]{enumitem}
\usepackage{mathrsfs}
\usepackage{amsthm}
\usepackage{tikz}
    \usetikzlibrary{matrix}
    \usetikzlibrary{decorations.markings}
    \usetikzlibrary{arrows}
\usepackage{hyperref}
\usepackage[T1]{fontenc}
\usepackage{oldgerm}
\usepackage{scalerel,stackengine}
\usepackage{stmaryrd}
\usepackage{xcolor}
\usepackage{xfrac}
\usepackage{MnSymbol}

\newcommand{\ad}{\textrm{\normalfont ad}}
\newcommand{\Ad}{\textrm{\normalfont Ad}}

\newtheorem{theorem}{Theorem}[section]
\newtheorem{definition}{Definition}[section]
\newtheorem{example}{Example}[section]
\newtheorem{proposition}{Proposition}[section]
\newtheorem{lemma}{Lemma}[section]
\newtheorem{corollary}{Corollary}[section]
\newtheorem{remark}{Remark}[section]

\begin{document}

\title[]{Higher-order Euler--Poincaré field equations}

\author{Marco Castrillón López}
\email{mcastri@mat.ucm.es}
\address{Facultad de Ciencias Matemáticas, Universidad Complutense de Madrid, Plaza de las Ciencias, 3, Madrid, 28040, Madrid, Spain}

\author{Álvaro Rodríguez Abella}
\address{Department of Mathematics and Computer Science, Saint Louis University (Madrid Campus), Avenida del Valle, 34, Madrid, 28003, Madrid, Spain}
\address{Instituto de Ciencias Matem\'aticas (CSIC--UAM--UC3M--UCM), Calle Nicol\'as Cabrera, 13--15, Madrid, 28049, Spain}
\email{alvrod06@ucm.es}

%\date{}

\begin{abstract}
We develop a reduction theory for $G$-invariant Lagrangian field theories defined on the higher-order jet bundle of a principal $G$-bundle, thus obtaining the higher-order Euler--Poincaré field equations. To that end, we transfer the Hamilton's principle to the reduced configuration bundle, which is identified with the bundle of flat connections (up to a certain order) of the principal $G$-bundle. As a result, the reconstruction condition is always satisfied and, hence, every solution of the reduced field equations locally comes from a solution of the original (unreduced) equations. Furthermore, the reduced equations are shown to be equivalent to the conservation of the Noether current. Lastly, we illustrate the theory by investigating multivariate higher-order splines on Lie groups.
\end{abstract}

\keywords{Euler--Poincaré equations, higher-order field theory, Lagrangian density, multivariate splines, Noether theorem} 
\subjclass{70S05, 70S10, 70H50 (Primary), 53C05, 58E15 (Secondary).}

\maketitle

%%%%%%%%%%%%%%%%%%%%%%%%%%%%%%%%%%%%%%%%%%%%%%%%%%%%%%%
%\tableofcontents

%%%%%%%%%%%%%%%%%%%%%%%%%%%%%%%%%%%%%%%%%%%%%%%%%%%%%%%
\section{Introduction}

In the realm of geometric mechanics, reduction by symmetries involves dropping the dynamical equations of a system with symmetry to a lower-dimensional space, which is the family of group orbits of the symmetry action. The modern approach of this procedure was first introduced in \cite{Ar1966,Sm1970,Me1973,MaWe1974}, and for Lagrangian theories in \cite{MaSh1993,CeMaRa2001,CeMaRa2001b}. For the latter, the key idea is to transfer the variational principle, typically the Hamilton's principle, to the reduced configuration space. This yields the reduced equations when applied to the reduced Lagrangian. 

The reduction results in mechanics for first-order sistems, i.e., those whose Lagrangians depend only on the generalized positions and their velocities, have been successfully extended to field theories~\cite{CaMa2008,GaRa2014}, including covariant field theories~\cite{CaGaRa2001,CaRa2003,ElGaHoRa2011,CaGaRo2013}, where the configuration manifold is replaced by a fiber bundle whose base space models the spacetime. The key case for reduction in both Mechanics and field theories is Euler--Poincaré reduction, where the variational variables take values in a Lie group $G$ or are local sections of a $G$-principal bundle, and $G$ is at the same time the group of symmetries.

Lagrangian reduction in Mechanics (and, in particular, Euler--Poincaré) has been extended to higher-order systems (see for example \cite{GaHoRa2011,GaHoMeRaVi2012,GaHoMeRaVi2012b}). However, there is not a counterpart in field theory yet, despite the fact that there are interesting systems and situations that fit in this context. 

This work addresses higher-order Euler--Poincaré reduction, i.e., reduction for higher-order Lagrangian field theories whose configuration bundle is a principal $G$-bundle and the group of symmetries is $G$. We thus extend the results in \cite{CaGaRa2001} to higher-order Lagrangian field theories. The analysis of  multivariate splines on Lie groups is given at the end of the article to illustrate the main results of it.

Reduction theory has also been analyzed in the discrete setting for both mechanics~\cite{BoSu1999a,MaPeSh1999,MaPeSh2000,BlCoJi2019,RoLe2023} and field theory~\cite{Va2007}. A discretization of the reduction theory introduced here would be desirable in order to build variational integrators for the reduced equations. We leave this as a future work.

The paper is organised as follows. Firstly, some essential facts about higher-order jet bundles are recalled in Section \ref{sec:higherorderjets}, including the total partial derivatives and the jet prolongation of vector fields, and the formulation of higher-order Lagrangian field theories is recalled in \ref{sec:calculus}. Next, in Section \ref{sec:reducedspace} the geometry of the reduced configuration space $(J^kP)/G$ for a $k$-th order Lagrangian theory on a principal bundle is investigated. The understanding of this reduced phase bundle is essential for a correct analysis of the reduced variational principle. For first order theories, this reduced bundle is the bundle of connections $C(P)$, the sections of which are principal connections on $P$. However, even though one might suspect that $(J^kP)/G\simeq J^{k-1}C(P)$, $k\geq 2$, it turns out that this reduced phase bundle is the space of $(k-1)$-th order pointwise principal connections with vanishing curvature up to $k-2$ order. This fact is already shown for particular cases in \cite{Jan04} and is analyzed in depth in \cite{EtMu1994}. We follow the strategy of the latter. The main result of the paper is established in Section \ref{sec:redeqs}, where we compute the reduced equations, i.e., the higher-order Euler--Poincaré field equations. For $k\geq 2$ the reconstruction process providing solutions of the original problem from solution of the reduced equations is given in detail (at least locally). In this case, the situation $k=1$ and $k\geq 2$ show a different conceptual behaviour, since the first order case requires an auxiliary compatibility condition, whereas that condition is implicit in the set of admissible solution for higher-order Lagrangians. Furthermore, the reduced equations are reinterpreted as a Noether conservation law in Section \ref{sec:noether}. At last, Section \ref{sec:splines} is devoted to illustrate the previous theory by computing the reduced equations for multivariate $k$-splines (brane splines) on Lie groups.

In the following, every manifold or map is assumed to be smooth, meaning $C^\infty$, unless otherwise stated. In addition, every fiber bundle $\pi_{Y,X}:Y\rightarrow X$ is assumed to be locally trivial and is denoted by $\pi_{Y,X}$. Given $x\in X$, $Y_x=\pi_{Y,X}^{-1}(\{x\})$ denotes the fiber over $x$. We assume that $\dim X=n$ and $\dim Y_x=m$. The space of (smooth) global sections of $\pi_{Y,X}$ is denoted by $\Gamma(\pi_{Y,X})$. In particular, vector fields on a manifold $X$ are denoted by $\mathfrak X(X)=\Gamma(\pi_{TX,X})$, where $TX$ is the tangent bundle of $X$. Likewise, $k$-forms on $X$ are denoted by $\Omega^k(X)=\Gamma(\pi_{T^*X,X})$, where $T^*X$ is the cotangent bundle of $X$. The space of local sections is denoted by $\Gamma_{loc}(\pi_{Y,X})$, and the same notation stands for local vector fields and forms. In the same vein, given an open set $\mathcal U\subset X$, the family of sections of $\pi_{Y,X}$ defined on $\mathcal U$ is denoted by $\Gamma\left(\mathcal U,\pi_{Y,X}\right)$, and analogous for the other spaces. The tangent map of a map $f\in C^\infty(X,X')$ between the manifolds $X$ and $X'$ is denoted by $(df)_x:T_xX\rightarrow T_{f(x)}X'$ for each $x\in X$. In the same vein, the pull-back of $\alpha\in\Omega^k(X')$ is denoted by $f^*\alpha\in\Omega^k(X)$ and its exterior derivative is denoted by ${\rm d}\alpha\in\Omega^{k+1}(X')$. When working in local coordinates, indices will be denoted by lowercase letters ($\mu$, $\alpha$, etc.), and multi-indices by capital letters ($J=(J_1,\dots,J_n)\in \mathbb{N}^n$, etc.). Given two multi-indices $J=(J_1,\dots,J_n),I=(I_1,\dots,I_n)\in \mathbb{N}^n$, we write
\begin{equation*}
J!=J_1!\dots J_n!,\qquad\binom{J}{I}=\binom{J_1}{I_1}\dots\binom{J_n}{I_n},
\end{equation*}
Besides, we will assume the Einstein summation convention for repeated (multi-)indices.

%%%%%%%%%%%%%%%%%%%%%%%%%%%%%%%%%%%%%%%%%%%%%%%%%%%%%%%
\section{Higher-order jet bundles}\label{sec:higherorderjets}

We summarize the main results about higher-order jet bundles that we will need in the following (see, for example, \cite[Chapter 6]{Sa1989}). Let $\pi_{Y,X}:Y\to X$ be a fiber bundle and $k\in\mathbb Z^+$. The $k$-th order jet bundle of $\pi_{Y,X}$ is denoted by
\begin{equation*}
\pi_{J^k Y,X}:J^k Y\longrightarrow X,    
\end{equation*}
and its elements by $j_x^k s$. The $k$-th jet lift of a section $s\in\Gamma(\pi_{Y,X})$ is denoted by $j^k s\in\Gamma\left(\pi_{J^k Y,X}\right)$. Recall that the maps $\pi_{k,l}:J^k Y\to J^l Y$, $j^k_xs\mapsto j^l_xs$, $0\leq l<k$, are fiber bundles, where we denote $J^0 Y=Y$. In addition, $\pi_{k,k-1}$ is an affine bundle modelled on
\begin{equation*}
\pi_{J^{k-1}Y,X}^*\left(\textstyle\bigvee^k T^*X\right)\otimes\pi_{k-1,0}^*(VY),
\end{equation*}
where $VY=\ker(\pi_{Y,X})_*$ is the vertical bundle of $\pi_{Y,X}$.

Let $(x^\mu,y^\alpha)$ be bundle coordinates for $\pi_{Y,X}$. The induced coordinates on $J^k Y$ are $(x^\mu,y^\alpha;y_J^\alpha)$, $1\leq|J|\leq k$, where $|J|=J_1+\dots+J_n$ is the length of $J$.  If, locally, $s(x^\mu)=\left(x^\mu,s^\alpha(x^\mu)\right)$ for some (local) functions $s^\alpha\in C^\infty(X)$, $1\leq\alpha\leq m$, then $j^k s(x^\mu)=\left(x^\mu,s^\alpha(x^\mu);s_J^\alpha(x^\mu)\right)$ with
\begin{equation*}
s_J^\alpha(x^\mu)=\left.\frac{\partial^{|J|}}{\partial x^J}\right|_{x=(x^\mu)}s^\alpha(x)=\left.\left(\frac{\partial}{\partial x^1}\right)^{J_1}\dots\left(\frac{\partial}{\partial x^n}\right)^{J_n}\right|_{x=(x^\mu)}s^\alpha(x),\qquad 1\leq|J|\leq k.
\end{equation*}

\begin{definition}
Let $\pi_{Y,X}$ and $\pi_{Y',X'}$ be fiber bundles, $k\in\mathbb Z^+$, and $F:Y\to Y'$ be a bundle morphism covering a diffeomorphism $f:X\to X'$. The \emph{$k$-th order jet lift of $F$} is the map $j^kF:J^k Y\to J^k Y'$ defined as $j_x^k s\mapsto j_{f(x)}^k\left(F\circ s\circ f^{-1}\right)$.
\end{definition}

If $\pi_{Y'',X''}:Y''\to X''$ is another fiber bundle and $G:Y'\to Y''$ is bundle morphism covering a diffeomorphism $g:X'\to X''$, then $j^k(G\circ F)=j^k G\circ j^k F$. In other words, the following diagram is commutative.
\begin{equation}\label{eq:commutativityjetlift}
\begin{tikzpicture}
\matrix (m) [matrix of math nodes,row sep=2em,column sep=4em,minimum width=2em]
{	J^k Y & J^k Y' & J^kY''\\
	 Y & Y'& Y''\\
	 X & X' & X''\\};
\path[-stealth]
(m-1-1) edge [] node [above] {$j^k F$} (m-1-2)
(m-1-1) edge [] node [left] {$\pi_{k,0}$} (m-2-1)
(m-1-2) edge [] node [right] {$\pi_{k,0}'$} (m-2-2)
(m-2-1) edge [] node [above] {$F$} (m-2-2)
(m-2-1) edge [] node [left] {$\pi_{Y,X}$} (m-3-1)
(m-2-2) edge [] node [right] {$\pi_{Y',X'}$} (m-3-2)
(m-3-1) edge [] node [above] {$f$} (m-3-2)
(m-1-2) edge [] node [above] {$j^k G$} (m-1-3)
(m-2-2) edge [] node [above] {$G$} (m-2-3)
(m-3-2) edge [] node [above] {$g$} (m-3-3)
(m-1-3) edge [] node [right] {$\pi_{k,0}''$} (m-2-3)
(m-2-3) edge [] node [right] {$\pi_{Y'',X''}$} (m-3-3)
(m-1-1) edge [bend left] node [above] {$j^k(F\circ G)$} (m-1-3);
\end{tikzpicture}    
\end{equation}

\begin{proposition}
The following is a canonical decomposition of vector bundles over $J^k Y$,
\begin{equation*}
\pi_{k+1,k}^*\left(T(J^k Y)\right)=\pi_{k+1,k}^*\left(V(J^k Y)\right)\oplus H(\pi_{k+1,k}),
\end{equation*}
where $V(J^k Y)=\ker(\pi_{J^k Y,X})_*$ is the vertical bundle of $\pi_{J^k Y,X}$ and $H(\pi_{k+1,k})_{j_x^{k+1}s}=d(j^k s)_x(T_x X)$ for each $j_x^{k+1}s\in J^{k+1}Y$. 
\end{proposition}

Given $j_x^{k+1} s\in J^{k+1}Y$ and $U_x\in T_x X$, the vector $U_{j_x^{k+1}s}^k=d(j^k s)_x(U_x)\in H(\pi_{k+1,k})$ is called \emph{$k$-th holonomic lift of $U_x$ by $j_x^{k+1}s$}. In coordinates, if we write $j_x^{k+1}s=\left(x^\mu,y^\alpha;y_J^\alpha\right)$ and $U_x=U^\mu\partial_\mu$, then
\begin{equation*}
U_{j_x^{k+1}s}^k=U^\mu\left(\partial_\mu+y_{1_\mu}^\alpha\partial_\alpha+\sum_{|J|=1}^k y_{J+1_\mu}^\alpha \partial_\alpha^J\right),
\end{equation*}
where $1_\mu$ is the multi-index given by $(1_\mu)_\nu=\delta_{\mu\nu}$, and $\partial_\mu$, $\partial_\alpha$ and $\partial_\alpha^J$ are the partial derivatives of the coordinates $x^\mu$, $y^\alpha$, $y_J^\alpha$, respectively, $1\leq\mu,\nu\leq n$, $1\leq\alpha\leq m$, $1\leq|J|\leq k$. In particular, the $k$-th holonomic lifts of the partial vector fields, $\partial_\mu$, $1\leq\mu\leq n$, are called \emph{coordinate total derivatives} and are given by
\begin{equation*}
\frac{d}{dx^\mu}=(\partial_\mu)_{j_x^{k+1}s}^k=\partial_\mu+y_{1_\mu}^\alpha\partial_\alpha+y_{J+1_\mu}^\alpha\partial_\alpha^J.
\end{equation*}

\begin{example}[Total time derivative]
Let $Q$ be a smooth manifold and consider the trivial bundle $\pi_{\mathbb R\times Q,\mathbb R}$ with (local) coordinates $(t,q^\alpha)$. Given $q(t)=\left(t,q^\alpha(t)\right)$, the \emph{total time derivative} is given by
\begin{equation*}
\frac{d}{dt}=\frac{\partial}{\partial t}+\frac{\partial q^\alpha}{\partial t}\partial_\alpha+\sum_{j=1}^k \frac{\partial^{j+1} q^\alpha}{\partial t^{j+1}}\partial_\alpha^j
\end{equation*}
\end{example}

Note that given a (local) function $f\in C^\infty(J^k Y)$ and a multi-index $J=(J^1,\dots,J^n)$, then
\begin{equation*}
\frac{d^{|J|}f}{dx^J}\in C^\infty(J^{k+|J|}Y).
\end{equation*}
Furthermore, for a section $s\in\Gamma(\pi_{Y,X})$ we have
\begin{equation}\label{eq:totalderivativesection}
\frac{d^{|J|}f}{dx^J}\circ j^{k+|J|}s=\frac{\partial^{|J|}\left(f\circ j^k s\right)}{\partial x^J}.
\end{equation}

Next, we define the prolongation of vector fields.

\begin{definition}
Let $l\in\mathbb Z^+$. A \emph{generalized vector field} on $J^lY$ is a section $U\in\Gamma\left(\pi_{l,0}^*TY\to J^l Y\right)$. Furthermore, it is said to be vertical if $U(j_x^l s)\in V_{s(x)}Y$ for each $j_x^l s\in J^l Y$.
\end{definition}

Note that (standard) vector fields on $Y$ would be  regarded as a generalized vector fields for $l=0$. Locally, generalized vector fields are given by
\begin{equation*}
U=U^\mu\partial_\mu+U^\alpha\partial_\alpha,\qquad U^\mu,U^\alpha:J^l Y\to\mathbb R,~1\leq\mu\leq n,~1\leq\alpha\leq m.
\end{equation*}
This way, vertical generalized vector fields are those ones such that $U^\mu=0$ for $1\leq\mu\leq n$.

Given a generalized vector field, $U\in\Gamma\left(\pi_{l,0}^*TY\to J^l Y\right)$, and $k\in\mathbb Z^+$, it may be defined the \emph{$k$-th order prolongation of $U$}, which is a section $U^{(k)}\in\Gamma\left(\pi_{k+l,k}^*\left(T\left(J^k Y\right)\right)\to J^{k+l}Y\right)$ (cf. \cite[Definition 6.4.16]{Sa1989}). In particular, $k$-the prolongations of (standard) vector fields on $Y$ are (standard) vector fields on $J^k Y$.

\begin{proposition}\label{prop:jetprolongationvector}
Let $k\in\mathbb Z^+$, $(x^\mu,y^\alpha;y_J^\alpha)$ be bundle coordinates for $J^k Y$ and $V=V^\alpha\partial_\alpha\in\mathfrak X(Y)$ be a vertical vector field on $Y$, where $V^\alpha:Y\to\mathbb R$, $1\leq\alpha\leq m$. Then $k$-th order prolongation of $V$ is given by
\begin{equation*}
V^{(k)}=V^\alpha\partial_\alpha+\frac{d^{|J|}V^\alpha}{dx^J}\partial_\alpha^J\in\mathfrak X\left(J^k Y\right).
\end{equation*}
\end{proposition}

\begin{proposition}\label{prop:prolongationflow}
Let $k\in\mathbb Z^+$ and $U\in\mathfrak X(Y)$ be a $\pi_{Y,X}$-projectable vector field on $Y$. Then the flow of its $k$-th order prolongation, $U^{(k)}\in\mathfrak X\left(J^k Y\right)$, is $\left\{j^k\phi_t\mid t\in(-\epsilon,\epsilon)\right\}$, where $\{\phi_t\mid t\in(-\epsilon,\epsilon)\}$ is the flow of $U$.
\end{proposition}

To conclude this brief overview, we present the Leibniz rule for higher-order, multivariable calculus, which can be straightforwardly proven by induction in the number of factors (see \cite[Proposition 6]{Ha2006} for two factors).

\begin{lemma}[Higher-order Leibniz rule]\label{lemma:higherleibniz}
Let $m\in\mathbb Z^+$ and $f_1,\dots,f_m\in C^\infty(\mathbb R^n)$. Then
\begin{equation*}
\frac{\partial^{|I|}}{\partial x^I}\prod_{\alpha=1}^m f_\alpha=\sum_{I^{(1)}+\dots+I^{(m)}=I}\frac{I!}{I^{(1)}!\dots I^{(m)}!}\prod_{\alpha=1}^m\frac{\partial^{|I^{(\alpha)}|}f_\alpha}{\partial x^{I^{(\alpha)}}},\qquad |I|\geq 0.
\end{equation*}
\end{lemma}

\begin{comment}
\begin{proof}
We show the formula by induction in $m\in\mathbb Z^+$. For $m=2$, the result was given in \cite[Proposition 6]{Ha2006}. Now, given $m>2$, we assume that the result holds for $m-1$, whence
\begin{align*}
\frac{\partial^{|I|}}{\partial x^I}\prod_{\alpha=1}^m f_\alpha & =\frac{\partial^{|I|}}{\partial x^I}\left(\left(\prod_{\alpha=1}^{m-1} f_\alpha\right)f_m\right)=\sum_{I^{(1)}+I^{(2)}=I}\frac{I!}{I^{(1)}!I^{(2)}!}\frac{\partial^{|I^{(1)}|}}{\partial x^{I^{(1)}}}\left(\prod_{\alpha=1}^{m-1} f_\alpha\right)\frac{\partial^{|I^{(2)|}}f_m}{\partial x^{I^{(2)}}}\\
& =\sum_{I^{(1)}+I^{(2)}=I}\frac{I!}{I^{(1)}!I^{(2)}!}\sum_{J^{(1)}+\dots+J^{(m-1)}=I^{(1)}}\left(\frac{I^{(1)}!}{J^{(1)}!\dots J^{(m-1)}!}\prod_{\alpha=1}^{m-1}\frac{\partial^{|J^{(\alpha)}|}f_\alpha}{\partial x^{J^{(\alpha)}}}\right)\frac{\partial^{|I^{(2)|}}f_m}{\partial x^{I^{(2)}}}\\
& =\sum_{I^{(1)}+I^{(2)}=I}\sum_{J^{(1)}+\dots+J^{(m-1)}=I^{(1)}}\frac{I!}{J^{(1)}!\dots J^{(m-1)}!I^{(2)}!}\prod_{\alpha=1}^{m-1}\left(\frac{\partial^{|J^{(\alpha)}|}f_\alpha}{\partial x^{J^{(\alpha)}}}\right)\frac{\partial^{|I^{(2)|}}f_m}{\partial x^{I^{(2)}}}\\
& =\sum_{J^{(1)}+\dots+J^{(m)}=I}\frac{I!}{J^{(1)}!\dots J^{(m)}!}\prod_{\alpha=1}^m\frac{\partial^{|J^{(\alpha)}|}f_\alpha}{\partial x^{J^{(\alpha)}}},
\end{align*}
where we have written $J^{(m)}=I^{(2)}$ in the last line.
\end{proof}
\end{comment}

%%%%%%%%%%%%%%%%%%%%%%%%%%%%%%%%%%%%%%%%%%%%%%%%%%%%%%%
\section{Calculus of variations for higher-order Lagrangian densities}\label{sec:calculus}

We now recall higher-order calculus of variations (for a comprehensive exposition see, for example,  \cite{FeFr1983,Ma1985,BrKr2005}). Let $X$ be a compact, $\pi_{Y,X}:Y\to X$ be a fiber bundle, and $k\in\mathbb Z^+$. A \emph{$k$-th order Lagrangian density} on $\pi_{Y,X}$ is a bundle morphism
\begin{equation*}
\mathfrak L:J^k Y\longrightarrow\textstyle\bigwedge^{n} T^*X
\end{equation*}
covering the identity on $X$. If $X$ is oriented by a volume form $v\in\Omega^n(X)$, we will write $\mathfrak L=Lv$ for a function $L\in C^\infty\left(J^k Y\right)$ known as \emph{Lagrangian}. The \emph{action functional} defined by $\mathfrak L$ is
\begin{equation*}
\mathbb S:\Gamma(\pi_{Y,X})\to\mathbb R,\qquad s\mapsto\int_X\,\mathfrak L\left(j^ks\right).
\end{equation*}
A \emph{variation} of $s\in\Gamma\left(\pi_{Y,X}\right)$ is a 1-parameter family of sections $\left\{s_t\in\Gamma\left(\pi_{Y,X}\right)\mid t\in(-\epsilon,\epsilon),~s_0=s\right\}$. The corresponding \emph{infinitesimal variation} is
\begin{equation*}
\delta s=\left.\frac{d}{dt}\right|_{t=0}s_t\in\Gamma\left(\pi_{s^*TY,X}\right).
\end{equation*}
Henceforth, only $\pi_{Y,X}$-vertical variations are considered, that is, those ones satisfying $\delta s(x)\in V_{s(x)}Y$ for each $x\in X$. The variation of the action functional induced by $\{s_t\}$ is defined as
\begin{equation*}
\delta\mathbb S[s]=\left.\frac{d\,\mathbb S[s_t]}{dt}\right|_{t=0},
\end{equation*}
and it only depends on the infinitesimal variation. That is to say, if $\{s_t\}$ and $\{s'_t\}$ are two variations of $s$ such that $\delta s=\delta s'$, then $\left.d\,\mathbb S[s_t]/dt\right|_{t=0}=\left.d\,\mathbb S[s'_t]/dt\right|_{t=0}$.

\begin{definition}
A section $s\in\Gamma\left(\pi_{Y,X}\right)$ is \emph{critical} or \emph{stationary for $\mathbb S$} if the variation of the corresponding action functional vanishes for every vertical variation of $s$, i.e.,
\begin{equation*}
\delta\mathbb S[s]=0,\qquad\delta s\in\Gamma\left(\pi_{s^*VY,X}\right).
\end{equation*}
\end{definition}

Unlike the first order case, the \emph{covariant Cartan form},
\begin{equation*}
\Theta_{\mathfrak L}\in\Omega^n\left(J^{2k-1} Y\right),    
\end{equation*}
is not uniquely defined for higher-order Lagrangians if $\dim X=n>1$. The main reason is that, although there is a canonical embedding $J^k Y\hookrightarrow J^1\left(J^{k-1}Y\right)$, there are many different choices for the corresponding projection. Nevertheless, it is always possible to construct a globally defined projection by means of tubular neighborhoods, thus yielding a globally defined covariant Cartan form on $J^k Y$ (cf. \cite[Theorem 6.5.13]{Sa1989}). In any case, there is a unique choice for the \emph{Euler--Lagrange form} associated to $\mathfrak L$,
\begin{equation*}
\mathcal{EL}(\mathfrak L)\in\Omega^n\left(J^{2k}Y,\pi_{2k,0}^*\left(V^*Y\right)\right),
\end{equation*}
where $\pi_{V^*Y,Y}$ is the dual of the vertical bundle $\pi_{VY,Y}$. There exist adapted coordinates $(x^\mu,y^\alpha;y_J^\alpha)$ for $J^{2k}Y$ such that the covariant Cartan form is locally given by\footnote{Since the higher-order covariant Cartan form is not uniquely determined, its local expression depends on the choice of bundle coordinates (cf. \cite[\S 5.5]{Sa1989}, \cite[\S 3B]{GoIsMaMo1997} for first order, and \cite[\S 5]{LeRo1987}, \cite[\S 6.5]{Sa1989} for higher orders).}
\begin{equation}\label{eq:Cartanformlocal}
\Theta_{\mathfrak L}=\sum_{|I|=0}^{k-1}\sum_{|J|=0}^{k-|I|-1}(-1)^{|J|} \frac{d^{|J|}}{dx^J}\left(\frac{\partial L}{\partial y_{I+J+1_\mu}^\alpha}\right)\\
\left(dy_I^\alpha-y_{I+1_\nu}^\alpha dx^\nu\right)\wedge\left(\iota_{\partial_\mu}v\right)+Lv,
\end{equation}
where $\iota:\mathfrak X(X)\times\Omega^{k+1}(X)\to\Omega^k(X)$ denotes the left interior product. Similarly, the Euler--Lagrange form is given by 
\begin{equation}\label{eq:ELformlocal}
\mathcal{EL}(\mathfrak L)=\sum_{|J|=0}^k(-1)^{|J|}\frac{d^{|J|}}{dx^J}\left(\frac{\partial L}{\partial y_J^\alpha}\right)v\otimes dy^\alpha.
\end{equation}

\begin{proposition}\label{prop:actionvariation}
Let $\delta s\in\Gamma(\pi_{s^*VY,X})$ be a variation of a section $s\in\Gamma(\pi_{Y,X})$. Then the variation of the action functional is given by
\begin{equation*}
\delta\mathbb S[s]=\int_X\left\langle\left(j^{2k}s\right)^*\mathcal{EL}(\mathfrak L),\delta s\right\rangle,
\end{equation*}
where $\langle\cdot,\cdot\rangle$ denotes the dual pairing within the vector bundle $VY\to Y$.
\end{proposition}

\begin{proof}
Since the variation is vertical, we have $s_t=\phi_t\circ s$, $t\in(-\epsilon,\epsilon)$, where $\left\{\phi_t:Y\to Y\mid t\in(-\epsilon,\epsilon)\right\}$ is the flow of a vertical, $\pi_{Y,X}$-projectable vector field $V\in\mathfrak X(Y)$. 
%Observe that $\phi_t:Y\to Y$ is an automorphism over the identity, $\operatorname{id}_X$, for each $t\in(-\epsilon,\epsilon)$. 
Subsequently,
\begin{equation}\label{eq:deltaS1}
\left.\frac{d}{dt}\right|_{t=0}\mathfrak L\left(j^k s_t\right)=\left.\frac{d}{dt}\right|_{t=0}\left(j^k s\right)^*\left(\mathfrak L(j^k\phi_t)\right)=\left(j^k s\right)^*\left(\left.\frac{d}{dt}\right|_{t=0}\mathfrak L\left(j^k\phi_t\right)\right)=\left(j^k s\right)^*\left(\pounds_{V^{(k)}}\mathfrak L\right),
\end{equation}
where $\pounds$ denotes the Lie derivative and we have used the commutativity of \eqref{eq:commutativityjetlift} and Proposition \ref{prop:prolongationflow}. Since $X$ is boundaryless, the Stokes theorem leads to
\begin{equation}\label{eq:deltaS2}
\int_X \left(j^k s\right)^*{\rm d}\left(\iota_{V^{(k)}}\mathfrak L\right)=\int_X {\rm d}\left(\left(j^k s\right)^*\left(\iota_{V^{(k)}}\mathfrak L\right)\right)=\int_{\partial X}\left(j^k s\right)^*\left(\iota_{V^{(k)}}\mathfrak L\right)=0.
\end{equation}
Similarly, we pick bundle coordinates $(x^\mu,y^\alpha;y_J^\alpha)$ for $J^k Y$, which allows us to write $V=V^\alpha\partial_\alpha$ for certain (local) functions $V^\alpha:X\to\mathbb R$, $1\leq\alpha\leq m$. Thus, Proposition \ref{prop:jetprolongationvector} yields
\begin{equation}\label{eq:deltaS3}
\iota_{V^{(k)}}{\rm d}\mathfrak L=\left(\frac{\partial L}{\partial y^\alpha}V^\alpha+\frac{\partial L}{\partial y_J^\alpha}\frac{d^{|J|}V^\alpha}{d x^J}\right)v.
\end{equation}
Lastly, the higher-order integration by parts formula for boundaryless manifolds (see, for example, \cite[Lemma 4.5]{CaLeMa2010}), together with \eqref{eq:totalderivativesection}, give
\begin{align}\label{eq:deltaS4}
\int_X\left(j^k s\right)^*\left(\frac{\partial L}{\partial y_J^\alpha}\frac{d^{|J|}V^\alpha}{d x^J}\,v\right) & =\int_X\frac{\partial }{\partial y_J^\alpha}\left(L\circ j^k s\right)\,\frac{\partial^{|J|}V^\alpha}{\partial x^J}\,v\\\nonumber
& =(-1)^{|J|}\int_X\frac{\partial^{|J|}}{\partial x^J}\left(\frac{\partial }{\partial y_J^\alpha}\left(L\circ j^k s\right)\right)\,V^\alpha\,v\\\nonumber
& =(-1)^{|J|}\int_X\left(j^{2k}s\right)^*\left(\frac{d^{|J|}}{dx^J}\left(\frac{\partial L}{\partial y_J^\alpha}\right)V^\alpha\,v\right).
\end{align}
On the other hand, it is clear that $\delta s=\left.\left(d/dt\right)\right|_{t=0}\left(\phi_t\circ s\right)=V\circ s=V^\alpha\partial_\alpha\in\Gamma\left(\pi_{s^*VY,X}\right)$. Hence, form the local expression of the Euler--Lagrange form \eqref{eq:ELformlocal} we get
\begin{equation}\label{eq:deltaS5}
\left\langle\left(j^{2k}s\right)^*\mathcal{EL}(\mathfrak L),\delta s\right\rangle=\left(j^{2k}s\right)^*\left(\sum_{|J|=0}^k(-1)^{|J|}\frac{d^{|J|}}{dx^J}\left(\frac{\partial L}{\partial y_J^\alpha}\right)\right)V^\alpha\,v.
\end{equation}

By using Cartan's formula, i.e., $\pounds=d\circ\iota+\iota\circ d$, and by gathering the previous expressions, we finish
\begin{align*}
\delta\mathbb S[s] & =\left.\frac{d}{dt}\right|_{t=0}\int_X \mathfrak L\left(j^k s_t\right)\overset{\eqref{eq:deltaS1}}{=}\int_X \left(j^k s\right)^*\left(\pounds_{V^{(k)}}\mathfrak L\right)\\
& =\int_X \left(j^k s\right)^*\left(\iota_{V^{(k)}}{\rm d}\mathfrak L\right)+\int_X \left(j^k s\right)^*{\rm d}\left(\iota_{V^{(k)}}\mathfrak L\right)\overset{\eqref{eq:deltaS2},\eqref{eq:deltaS3}}{=}\int_X\left(j^k s\right)^*\left(\frac{\partial L}{\partial y^\alpha}V^\alpha+\frac{\partial L}{\partial y_J^\alpha}\frac{d^{|J|}V^\alpha}{d x^J}\right)v\\
& \overset{\eqref{eq:deltaS4}}{=}\int_X\left(j^{2k}s\right)^*\left(\frac{\partial L}{\partial y^\alpha}+(-1)^{|J|}\frac{d^{|J|}}{dx^J}\left(\frac{\partial L}{\partial y_J^\alpha}\right)\right)V^\alpha\,v\overset{\eqref{eq:deltaS5}}{=}\int_X\left\langle\left(j^{2k}s\right)^*\mathcal{EL}(\mathfrak L),\delta s\right\rangle.
\end{align*}
\end{proof}

The following result is a straightforward consequence of the previous Proposition and it gives the higher-order version of the well-known Euler--Lagrange equations.

\begin{theorem}\label{theorem:ELequations}
Let $\mathfrak L=Lv:J^k Y\to\bigwedge^n T^*X$ be a $k$-th order Lagrangian density, and consider the corresponding Euler--Lagrange form, $\mathcal{EL}(\mathfrak L)$. Then the following statements for a section $s\in\Gamma\left(\pi_{Y,X}\right)$ are equivalent:
\begin{enumerate}[(i)]
    \item The variational principle $ \delta\mathbb S[s]=0$ holds for arbitrary variations $\delta s\in\Gamma(\pi_{s^*VY,X})$.
    \item $s$ satisfies the \emph{$k$-th order Euler--Lagrange field equations}, i.e., $(j^{2k} s)^*\mathcal{EL}(\mathfrak L)=0$.
\end{enumerate}
\end{theorem}

From \eqref{eq:totalderivativesection} and \eqref{eq:ELformlocal}, we obtain the local expression of the higher-order Euler--Lagrange equations:
\begin{equation*}
\sum_{|J|=0}^k(-1)^{|J|}\frac{\partial^{|J|}}{\partial x^J}\left(\frac{\partial L}{\partial y_J^\alpha}\left(j^k s\right)\right)=0,\qquad 1\leq\alpha\leq m.
\end{equation*}
    
\begin{remark}\label{remark:noncompact}
The calculus of variations described above is straightforwardly extended to a non-compact base manifold $X$ by considering compactly supported variations. In other words, given a section $s\in\Gamma(\pi_{Y,X})$, the only variations $\delta s$ allowed are those satisfying 
\begin{equation*}
j_x^{k-1}\,\delta s=0,\qquad x\in\partial\mathcal U,
\end{equation*}
for some open subset $\mathcal U\subset X$ with compact closure, $\overline{\mathcal U}$. Locally, this condition ensures that $\delta s$ as well as its partial derivatives up to order $k-1$ vanish on the boundary, $\partial\mathcal U$. As a result, no boundary terms appear when integrating by parts and, thus, Proposition \ref{prop:actionvariation} is still valid.
\end{remark}

%%%%%%%%%%%%%%%%%%%%%%%%%%%%%%%%%%%%%%%%%%%%%%%%%%%%%%%
\section{Geometry of the reduced configuration space}\label{sec:reducedspace}

Let $G$ be a Lie group, $\mathfrak g=T_e G$ be its Lie algebra, $e\in G$ being the identity element, and $\mathfrak g^*$ be the dual of the Lie algebra. The corresponding exponential map is denoted by $\exp:\mathfrak g\to G$, and the adjoint representation of $G$ is denoted by $\Ad_g:\mathfrak g\to\mathfrak g$ for each $g\in G$. In the same vein, the adjoint and coadjoint representations of $\mathfrak g$ are denoted by $\ad_\xi:\mathfrak g\to\mathfrak g$ and $\ad_\xi^*:\mathfrak g^*\to\mathfrak g^*$, respectively, for each $\xi\in\mathfrak g$.

Let $\pi_{P,X}:P\to X$ be a principal $G$-bundle. The corresponding right action is denoted by
\begin{equation*}
R:P\times G\to P,\qquad(y,g)\mapsto R_y(g)=R_g(y)=y\cdot g.    
\end{equation*}
Recall that the infinitesimal generator of $\xi\in\mathfrak g$ is the vertical vector field $\xi^*\in\mathfrak X(P)$ given by
\begin{equation*}
\xi_y^*=\left.\frac{d}{dt}\right|_{t=0}y\cdot\exp(t\xi)=\left(dR_y\right)_e(\xi),\qquad y\in P.
\end{equation*}

For each $k\in\mathbb Z^+$, the $k$-th jet extension of the action,
\begin{equation*}
R^{(k)}:J^k P\times G\to J^k P,\qquad \left(j_x^k s,g\right)\mapsto j_x^k(s\cdot g),
\end{equation*}
is again free and proper, thus yielding a principal $G$-bundle, $J^k P\to J^k P/G$.

We denote by $C(P)=J^1 P/G\to X$ the bundle of connections of $\pi_{P,X}$, which is an affine bundle modelled on $T^*X\otimes \ad(P)\to X$, being $\ad(P)=(P\times\mathfrak g)/G$ the adjoint bundle of $\pi_{P,X}$. Recall that there is a bijective correspondence between (local) sections of $\pi_{C(P),X}$ and (local) principal connections on $\pi_{P,X}$, which we denote by
\begin{equation*}
\Omega^1(P,\mathfrak g)\ni{A}\overset{1:1}{\longleftrightarrow}\sigma_{A}\in\Gamma(\pi_{C(P),X}).
\end{equation*}
The following result ensures that holonomic sections of the jet bundle yield flat principal connections.

\begin{lemma}\label{lemma:holonomicflat}
Let $s\in\Gamma_{loc}(\pi_{P,X})$ and define $\sigma_{A}=\left[j^1 s\right]_G\in\Gamma_{loc}(\pi_{C(P),X})$. Then the local connection ${A}\in\Omega_{loc}^1(P,\mathfrak g)$ is flat.
\end{lemma}

\begin{proof}
The local section $s$ induces a trivialization of $\pi_{P,X}$ on its domain $\mathcal U\subset X$,
\begin{equation*}
\varphi:\mathcal U\times G\overset{\sim}{\to}P|_{\mathcal U},\qquad(x,g)\mapsto \varphi(x,g)=s(x)\cdot g.
\end{equation*}
Recall that any connection is uniquely determined by its horizontal lift at each point. Given $y=s(x)\cdot g\in P|_{\mathcal U}$, we denote by $Hor_y^{A}: T_x X\to T_y P$ the horizontal lifting of ${A}$ at $y$. Since $\sigma_{A}=\left[j^1 s\right]_G$, we have
\begin{equation}\label{eq:horizontalliftsecion}
Hor_y^{A}(U_x)=d(R_g\circ s)_x(U_x)\in T_y P,\qquad U_x\in T_xX.
\end{equation}
On the other hand, consider the canonical flat connection ${A}_0\in\Omega^1(X\times G,\mathfrak g)$ on the trivial bundle $\pi_{X\times G,X}$. For each $(x,g)\in X\times G$ it is defined as
\begin{equation*}
Hor_{(x,g)}^{{A}_0}(U_x)=(U_x,0_g)\in T_{(x,g)}(X\times G),\qquad U_x\in T_x X.
\end{equation*}
It turns out that $\varphi^*{A}=\left.{A}_0\right|_{\mathcal U\times G}$, so ${A}$ is a flat connection. Indeed, let $(x,g)\in\mathcal U\times G$ and $U_x\in T_x X$. Since $\varphi$ covers the identity on $\mathcal U$, we have
\begin{align*}
Hor_{(x,g)}^{\varphi^*{A}}(U_x) & =\left(d\varphi^{-1}\right)_{s(x)\cdot g}\left(Hor_{\varphi(x,g)}^{A}(U_x)\right)=\left(d\varphi^{-1}\right)_{s(x)\cdot g}\left(d(R_g\circ s)_x(U_x)\right)\\
& =d\left(\varphi^{-1}\circ R_g\circ s\right)_x(U_x)=(d\iota_g)_x(U_x)=(U_x,0_g)=Hor_{(x,g)}^{{A}_0}(U_x),
\end{align*}
where $X\ni x\mapsto\iota_g(x)=(x,g)\in X\times G$.
\end{proof}

The main goal of this section is to study the geometry of the quotient $\left.\left(J^k P\right)\right/G$, which was first investigated in \cite{EtMu1994}. For the convenience of the reader, we present a detailed proof of the main result (see Corollary \ref{corollary:isomorphismquotient} below). To begin with, we introduce the following bundle morphisms.

\begin{proposition}\label{prop:welldefined}
For each $k\geq 2$, the bundle morphism $\Theta_k:J^k P\to J^{k-1}C(P)$ defined as $j_x^k s\mapsto j_x^{k-1}[j^1 s]_G$ is well-defined and it descends to a map on the quotient, $\Xi_k:\left.\left(J^k P\right)\right/G\to J^{k-1}C(P)$.
\end{proposition}

\begin{proof}
To begin with, note that $\Theta_k$ is the composition of the natural immersion $J^k P\hookrightarrow J^{k-1}\left(J^1 P\right)$ with the $(k-1)$-th jet lift of the canonical projection $\pi_1:J^1 P\to C(P)$. Thus, it is well-defined.

For the second part, let $g\in G$. Then
\begin{equation*}
\Theta_k\left((j_x^ks)\cdot g\right)=\Theta_k\left(j_x^k(s\cdot g)\right)=j_x^{k-1}[j^1 (s\cdot g)]_G=j_x^{k-1}[(j^1 s)\cdot g]_G=j_x^{k-1}[j^1 s]_G=\Theta_k\left(j_x^k s\right).
\end{equation*}
\end{proof}

The previous result may be summarized in the following commutative diagram.

\begin{equation}\label{eq:diagram}
\begin{array}{cc}
\begin{tikzpicture}
\matrix (m) [matrix of math nodes,row sep=3em,column sep=5em,minimum width=2em]
{	J^k P & J^{k-1}C(P)\\
	\left.\left(J^k P\right)\right/G & \\};
\path[-stealth]
(m-1-1) edge [] node [above] {$\Theta_k$} (m-1-2)
(m-1-1) edge [] node [left] {$\pi_k$} (m-2-1)
(m-2-1) edge [] node [below] {$~\Xi_k$} (m-1-2);
\end{tikzpicture} 
\quad & \quad
\begin{tikzpicture}
\matrix (m) [matrix of math nodes,row sep=3em,column sep=5em,minimum width=2em]
{	j_x^k s & j_x^{k-1}[j^1 s]_G \\
	\left[j_x^k s\right]_G & \\};
\path[-stealth]
(m-1-1) edge [|->] node [above] {} (m-1-2)
(m-1-1) edge [|->] node [left] {} (m-2-1)
(m-2-1) edge [|->] node [right] {} (m-1-2);
\end{tikzpicture}
\end{array}
\end{equation}

Consider the curvature map, i.e.,
\begin{equation*}
\tilde F:J^1C(P)\to\textstyle\bigwedge^2 T^*X\otimes \ad(P),\qquad j_x^1\sigma_{A}\mapsto\tilde F^A(x),
\end{equation*}
where $\tilde F^A\in\Omega^2(X,\ad(P))$ is curvature of the (local) principal connection ${A}\in\Omega^1(P,\mathfrak g)$ regarded as a 2-from on $M$ with values in the adjoint bundle. By lifting it to the $(k-2)$-jet and by composing it with the natural immersion of $J^{k-1}C(P)\hookrightarrow J^{k-2}\left(J^1 C(P)\right)$ we obtain
\begin{equation*}
j^{k-2}\tilde F:J^{k-1}C(P)\to J^{k-2}\left(\textstyle\bigwedge^2 T^*X\otimes \ad(P)\right),\qquad j_x^{k-1}\sigma_{A}\mapsto j_x^{k-2}\tilde F^A.
\end{equation*}
Let $\hat 0:X\to\bigwedge^2 T^*X\otimes \ad(P)$ be the zero section, and consider its $(k-2)$-jet lift, $j^{k-2}\hat 0:X\to J^{k-2}\left(\bigwedge^2 T^*X\otimes \ad(P)\right)$. We define the kernel of $j^{k-2}\tilde F$ as
\begin{equation*}
\ker j^{k-2}\tilde F=\left\{j_x^{k-1}\sigma_{A}\in J^{k-1}C(P)\mid j_x^{k-2}\tilde F^A=j_x^{k-2}\hat 0\right\}.
\end{equation*}

\begin{proposition}\label{prop:Psikproperties}
For each $k\geq 2$, we have
\begin{enumerate}[(i)]
    \item $\Xi_k$ is an injective bundle morphism over $X$.
    \item $\operatorname{im}\Xi_k\subset\ker j^{k-2}\tilde F$.
\end{enumerate}
\end{proposition}

\begin{proof}
To check \emph{(i)}, let $\left[j_x^k s\right]_G\neq\left[j_x^k s'\right]_G$ be two different elements in $\left.\left(J^k P\right)\right/G$. In other words, $j_x^k s\neq(j_x^k s')\cdot g=j_x^k(s'\cdot g)$ for each $g\in G$. Subsequently, $[j_x^1 s]_G\neq[j_x^1 s']_G$ and, thus, $\Xi_k\left(\left[j_x^k s\right]_G\right)\neq\Xi_k\left(\left[j_x^k s'\right]_G\right)$. On the other hand, \emph{(ii)} is a straightforward consequence of Lemma \ref{lemma:holonomicflat}.
\end{proof}

Let us introduce a trivializing chart of $\pi_{P,X}$, i.e., we pick an open subset $\mathcal U\subset X$ with $\mathcal U\simeq\mathbb R^n$ via the local coordinates $x=(x^\mu)$ and such that $P|_{\mathcal U}\simeq\mathcal U\times G$. To keep the notation simple, we write $\mathcal U=X$. Likewise, let $\{B_\alpha\in\mathfrak g\mid 1\leq\alpha\leq m\}$ be a basis of $\mathfrak g$ and consider normal coordinates $(y^\alpha)$ in a neighborhood of the identity element $e\in G$, i.e., $g=\exp(y^\alpha B_\alpha)$. This way, we have normal bundle coordinates $(x^\mu,y^\alpha)$ for $\pi_{P,X}$. Note that for each $g=\exp(y^\alpha B_\alpha),h=\exp(z^\alpha B_\alpha)\in G$ close enough to the identity, we have $gh=\exp\left(f_{CD}^\alpha y^C z^D B_\alpha\right)$,
where $C=(C_1,\dots,C_m)$ is a multi-index, $y^C=(y^1)^{C_1}\dots(y^m)^{C_m}$ and $f_{CD}^\alpha$, $1\leq\alpha\leq m$, $|C|,|D|\geq 0$, are the constants of the Baker--Campbell--Hausdorff formula (cf. \cite[\S 2.15]{Va1984} and \cite[Theorem 4.1]{EtMu1994}). By using this, it can be checked that
\begin{equation}\label{eq:xitopartiallocal}
P\times\mathfrak g\simeq VP,\qquad((x^\mu,y^\alpha),\xi^\alpha B_\alpha)\mapsto a_\beta^\alpha\xi^\beta\partial_\alpha,
\end{equation}
where
$$
a_\beta^\alpha(y^1,\ldots, y^m)=f_{C 1_{\beta}}^\alpha y^C.
$$
Note that for $y=(x,e)$ we have $a_\beta^\alpha(0,\ldots, 0)=\delta_\beta^\alpha$, $1\leq\alpha\leq m$. Hence, the inverse of the matrix $(a_\beta^\alpha(y^\alpha))$ is well defined in a neighbourhood of $y=(x,e)$, and we denote it by $(b_\beta^\alpha(y^1,\ldots, y^m))$. In such case, the inverse of the previous isomorphism is given by
\begin{equation}\label{eq:partialtoxilocal}
VP\simeq P\times\mathfrak g,\qquad (y,U^\alpha\partial_\alpha)\mapsto(y,b_\beta^\alpha U^\beta B_\alpha).
\end{equation}

We pick bundle coordinates $(x^\mu,{A}_\mu^\alpha)$ and $(x^\mu,F_{\mu\nu}^\alpha)$ for $\pi_{C(P),X}$ and $\pi_{\bigwedge^2 T^*X\otimes \ad(P),X}$, respectively. Similarly, we consider the corresponding bundle coordinates $(x^\mu,y^\alpha;y_J^\alpha)$, $(x^\mu,{A}_\mu^\alpha;{A}_{\mu,J}^\alpha)$ and $(x^\mu,F_{\mu\nu}^\alpha;F_{\mu\nu,J}^\alpha)$ for $J^k P$, $J^{k-1}C(P)$ and $J^{k-2}\left(\bigwedge^2 T^*X\otimes \ad(P)\right)$, respectively. In order to keep the notation simple, we will write $y_{J_0}^\alpha=y^\alpha$, ${A}_{\mu,J_0}^\alpha={A}_\mu^\alpha$ and $F_{\mu\nu,J_0}^\alpha=F_{\mu\nu}^\alpha$ for $J_0=(0,\dots,0)$.

\begin{remark}
Given two multi-indices $I$, $J$, the inequality $I\leq J$ means that $I_\mu\leq J_\mu$ for $1\leq\mu\leq n$. Likewise, we denote $I<J$ when $I\leq J$ and $I\neq J$, that is, $I_\mu\leq J_\mu$ for $1\leq\mu\leq n$ and there exists $\mu_0\in\{1,\dots,n\}$ such that $I_{\mu_0}<J_{\mu_0}$.
\end{remark}

Let us compute the coordinated expression of $\Theta_k$.

\begin{proposition}\label{prop:Thetakcoordinates}
For each $k\geq 2$ and each $(x^\mu,y^\alpha;y_J^\alpha)\in J^k P$ with $|y^\alpha|$ small enough, $1\leq\alpha\leq m$, we have $\Theta_k(x^\mu,y^\alpha;y_J^\alpha)=(x^\mu,{A}_\mu^\alpha;{A}_{\mu,J}^\alpha)$, where\footnote{
By abusing the notation, for each multi-index $J$ we denote
\begin{equation*}
\left.\frac{\partial^J b_\beta^\alpha}{\partial x^J}(y)\right|_x=\left.\frac{\partial b_\beta^\alpha}{\partial y^\gamma}(y)\right|_{y^\alpha=0}y_J^\gamma,\qquad1\leq\alpha,\beta\leq m.
\end{equation*}}
\begin{equation*}
{A}_\mu^\alpha=b_\beta^\alpha(y^\alpha)y_{1_\mu}^\beta,\quad{A}_{\mu,J}^\alpha=\left.\frac{\partial^{|J|}}{\partial x^J}\left(b_\beta^\alpha(y^\alpha)y_{1_\mu}^\beta\right)\right|_x,\quad1\leq\mu\leq n,~1\leq\alpha\leq m,~1\leq|J|\leq k-1.
\end{equation*}
\end{proposition}

\begin{proof}
We denote $j_x^1 s=(x^\mu,y^\alpha;y_J^\alpha)$. Recall from \eqref{eq:horizontalliftsecion} that the horizontal lift given by $[j^1 s]_G\in C(P)$ reads $Hor_{s(x)}^{A}=dx^\mu\otimes\partial_\mu+y_{1_\mu}^\alpha dx^\mu\otimes\partial_\alpha$. From this and \eqref{eq:partialtoxilocal}, we get
\begin{equation*}
{A}_{s(x)}=b_\beta^\alpha(y^\alpha)\left(-y_{1_\mu}^\beta dx^\mu+dy^\alpha\right)\otimes B_\alpha\in T_{s(x)}^*P\otimes\mathfrak g.
\end{equation*}
Then $[j^1 s]_G=(x^\mu,{A}_\mu^\alpha=b_\beta^\alpha(y^\alpha) y_{1_\mu}^\beta)$. By taking partial derivatives, we conclude.
\end{proof}

Now some technical lemmas are presented.

\begin{lemma}\label{lemma:curvaturecoordinates}
Let $k\geq 2$ and $j_x^{k-1}\sigma_{A}=(x^\mu,{A}_\mu^\alpha;{A}_{\mu,J}^\alpha)\in J^{k-1}C(P)$. Then $j_x^{k-2}\tilde F^A=\left(x^\mu,F_{\mu\nu}^\alpha;F_{\mu\nu,J}^\alpha\right)\in J^{k-2}\left(\bigwedge^2 T^*X\otimes \ad(P)\right)$ is given by
\begin{equation}\label{eq:curvaturelocal}
F_{\mu\nu,J}^\alpha=\frac{1}{2}\left({A}_{\nu,J+1_\mu}^\alpha-{A}_{\mu,J+1_\nu}^\alpha+\sum_{I\leq J}\binom{J}{I}c_{\beta\gamma}^\alpha{A}_{\mu,I}^\beta{A}_{\nu,J-I}^\gamma\right),
\end{equation}
for each $1\leq\mu,\nu\leq n$, $0\leq|J|\leq k-2$ and $1\leq\alpha\leq m$, where we denote by $c_{\beta\gamma}^\alpha=2 f_{1_\beta 1_\gamma}^\alpha$, $1\leq\alpha,\beta,\gamma\leq m$, the structure constants of $\mathfrak g$.
\end{lemma}

\begin{proof}
We show it by induction in $0\leq |J|\leq k-2$. The case $|J|=0$ is straightforward, since the curvature of $j_x^1\sigma_{A}=(x^\mu,{A}_\mu^\alpha;{A}_{\mu,1_\nu}^\alpha)$ is given by (cf. \cite[Chapter 1, \S 4]{Sa2013})
\begin{equation*}
F_{\mu\nu}^\alpha=\frac{1}{2}\left({A}_{\nu,1_\mu}^\alpha-{A}_{\mu,1_\nu}^\alpha+c_{\beta\gamma}^\alpha{A}_\mu^\beta{A}_\nu^\gamma\right).
\end{equation*}
Now assume that \eqref{eq:curvaturelocal} holds for $1\leq|J|(<k-2)$. To conclude, we need to show that it holds for $J'=J+1_\sigma$ for each $1\leq\sigma\leq n$. Note that
\begin{equation*}
\binom{J'}{I }=\binom{J+1_\sigma}{I }=\binom{J}{I }+\binom{J}{I -1_\sigma},\qquad 1_\sigma\leq I \leq J.
\end{equation*}
Furthermore, note that 
\begin{equation*}
\sum_{\substack{I\leq J,\\I_\sigma=0}}\binom{J'}{I}=\sum_{\substack{I\leq J,\\I_\sigma=0}}\binom{J}{I},\qquad\sum_{\substack{I\leq J',\\I_\sigma=J_\sigma+1}}\binom{J'}{I}=\sum_{\substack{I\leq J',\\I_\sigma=J_\sigma+1}}\binom{J}{I-1_\sigma}.
\end{equation*}
Hence,
\begin{equation*}
\sum_{I \leq J'}\binom{J'}{I }=\sum_{\substack{I\leq J,\\I_\sigma=0}}\binom{J}{I}+\sum_{1_\sigma\leq I \leq J}\left(\binom{J}{I }+\binom{J}{I -1_\sigma}\right)+\sum_{\substack{I\leq J',\\I_\sigma=J_\sigma+1}}\binom{J}{I-1_\sigma}=\sum_{I\leq J}\binom{J}{I}+\sum_{1_\sigma\leq I\leq J'}\binom{J}{I-1_\sigma}.
\end{equation*}
By using this and by taking the partial derivative of \eqref{eq:curvaturelocal} with respect to $x^\sigma$, we finish
\begin{align*}
F_{\mu\nu,J'}^\alpha & =\frac{1}{2}\left({A}_{\nu,J'+1_\mu}^\alpha-{A}_{\mu,J'+1_\nu}^\alpha+\sum_{I\leq J}\binom{J}{I}c_{\beta\gamma}^\alpha\left({A}_{\mu,I+1_\sigma}^\beta{A}_{\nu,J-I}^\gamma+{A}_{\mu,I}^\beta{A}_{\nu,J-I+1_\sigma}^\gamma\right)\right)\\
& =\frac{1}{2}\left({A}_{\nu,J'+1_\mu}^\alpha-{A}_{\mu,J'+1_\nu}^\alpha+\sum_{1_\sigma\leq I \leq J'}\binom{J}{I -1_\sigma}c_{\beta\gamma}^\alpha{A}_{\mu,I }^\beta{A}_{\nu,J'-I }^\gamma+\sum_{I\leq J}\binom{J}{I}c_{\beta\gamma}^\alpha{A}_{\mu,I}^\beta{A}_{\nu,J'-I}^\gamma\right)\\
& =\frac{1}{2}\left({A}_{\nu,J'+1_\mu}^\alpha-{A}_{\mu,J'+1_\nu}^\alpha+\sum_{I \leq J'}\binom{J'}{I }c_{\beta\gamma}^\alpha{A}_{\mu,I }^\beta{A}_{\nu,J'-I }^\gamma\right).
\end{align*}
\end{proof}

\begin{lemma}\label{lemma:liebracketbs}
Let $k\geq 2$ and $g=\exp(y^\alpha B_\alpha)\in G$ with $|y^\alpha|$ small enough, $1\leq\alpha\leq m$. Then
\begin{equation*}
c_{\beta\gamma}^\lambda b_\alpha^\beta(y)b_\epsilon^\gamma(y)=\frac{\partial b_\alpha^\lambda}{\partial y^\epsilon}(y)-\frac{\partial b_\epsilon^\lambda}{\partial y^\alpha}(y),\qquad1\leq\alpha,\epsilon,\lambda\leq m.
\end{equation*}
\end{lemma}

\begin{proof}
For each $1\leq\beta,\gamma\leq m$ and $y=(x^\mu,y^\alpha)\in P$, the infinitesimal generator is given in coordinates by $(B_\beta)_y^*=a_\beta^\alpha(y)\partial_\alpha$ and, from the equality $[B_\beta,B_\gamma]=c_{\beta\gamma}^\alpha B_\alpha$, they satisfy
\begin{equation}\label{eq:corchetescoordenadas}
c_{\beta\gamma}^\alpha a_\alpha^\delta(t)\partial_\delta=\left[a_\beta^\alpha(y)\partial_\alpha,a_\gamma^\delta(y)\partial_\delta\right]=\left(a_\beta^\alpha(y)\frac{\partial a_\gamma^\delta}{\partial y^\alpha}(y)-a_\gamma^\alpha\frac{\partial a_\beta^\delta(y^\alpha)}{\partial y^\alpha}(y)\right)\partial_\delta.
\end{equation}
Since $\left(a_\beta^\alpha(y)\right)\left(b_\beta^\alpha(y)\right)=\mathbb I$, the identity matrix, we have $a_\gamma^\delta(y)b_\delta^\beta(y)=\delta_\gamma^\beta$ and, thus,
\begin{equation*}
\frac{\partial a_\gamma^\delta}{\partial y^\alpha} (y)b_\delta^\beta(y)+a_\gamma^\delta(y)\frac{\partial b_\delta^\beta(y^\alpha)}{\partial y^\alpha}(y)=0,\qquad1\leq\alpha,\beta,\gamma\leq m.
\end{equation*}
As a result,
\begin{equation*}
\frac{\partial a_\gamma^\delta}{\partial y^\alpha}(y)=-a_\beta^\delta(y)a_\gamma^\epsilon(y)\frac{\partial b_\epsilon^\beta}{\partial y^\alpha}(y),\qquad1\leq\alpha,\gamma,\delta\leq m.
\end{equation*}
By substituting this into \eqref{eq:corchetescoordenadas} and rearranging indices, we get
\begin{equation*}
c_{\beta\gamma}^\kappa a_\kappa^\delta(y)=a_\beta^\alpha(y)a_\lambda^\delta(y)a_\gamma^\epsilon (y)\left(-\frac{\partial b_\epsilon^\lambda}{\partial y^\alpha}(y)+\frac{\partial b_\alpha^\lambda}{\partial y^\epsilon}(y)\right),\qquad1\leq\beta,\gamma,\delta\leq m.
\end{equation*}
Lastly, we multiply each side by $b_\alpha^{\beta'}(y)b_\delta^{\lambda'}(y)b_\epsilon^{\gamma'}(y)$ and we sum over $\alpha$, $\delta$ and $\epsilon$:
\begin{equation*}
c_{\beta\gamma}^\lambda b_\alpha^\beta(y)b_\epsilon^\gamma(y)=\frac{\partial b_\alpha^\lambda}{\partial y^\epsilon}(y)-\frac{\partial b_\epsilon^\lambda}{\partial y^\alpha}(y),\qquad1\leq\alpha,\epsilon,\lambda\leq m.
\end{equation*}
\end{proof}

We are ready to show the main result about the map $\Xi_k$.

\begin{theorem}\label{theorem:imageandproper}
For each $k\geq 2$, we have that:
\begin{enumerate}[(i)]
    \item $\operatorname{im}\Xi_k=\ker j^{k-2}\tilde F\subset J^{k-1}C(P)$ and,
    \item $\Xi_k$ is a proper map.
\end{enumerate}
\end{theorem}

\begin{proof}
To begin with, we show \emph{(i)}. Thanks to Proposition \ref{prop:Psikproperties}\emph{(ii)}, we only need to prove that $\ker j^{k-2}\tilde F\subset\operatorname{im}\Xi_k$. Similarly, the commutativity of \eqref{eq:diagram} yields $\operatorname{im}\Xi_k=\operatorname{im}\Theta_k$.

The fact that the maps are local enables us to work locally. More specifically, we employ the bundle coordinates introduced above. Let $j_x^{k-1}\sigma_{A}=(x^\mu,{A}_\mu^\alpha;{A}_{\mu,J}^\alpha)\in\ker j^{k-2}\tilde F$. We need to find $j_x^k s=(x^\mu,y^\alpha;y_J^\alpha)\in J^k P$, such that $\Theta_k(j_x^k s)=j_x^{k-1}\sigma_{A}$. For that matter, we set $y^\alpha=0$, $1\leq\alpha\leq m$. Besides, we define $y_J^\alpha$, $1\leq |J|\leq k$, by recurrence. Namely, fix $0\leq r\leq k-1$, and assume that $y_J^\alpha$ has been defined for $1\leq\alpha\leq m$, $0\leq|J|\leq r-1$ in such a way that they satisfy
\begin{equation}\label{eq:proof0}
\left.\frac{\partial^{|J|}}{\partial x^J}\left(b_\beta^\alpha(y)y_{1_\mu}^\beta\right)\right|_x={A}_{\mu,J}^\alpha,\qquad 0\leq|J|\leq r-2,~1\leq\mu\leq n,~1\leq\alpha\leq m.
\end{equation}
Now pick $J\in\mathbb N^n$ with $|J|=r$ and consider $1\leq\mu\leq n$ such that $J_\mu>0$. Then we set
\begin{equation*}
y_J^\alpha(\mu)={A}_{\mu,J-1_\mu}^\alpha-\sum_{I<J-1_\mu}\binom{J-1_\mu}{I}\left.\frac{\partial^{|J-1_\mu-I|}b_\beta^\alpha(y)}{\partial x^{J-1_\mu-I}}\right|_x y_{I+1_\mu}^\beta,\qquad1\leq\alpha\leq m.
\end{equation*}
Observe that it makes sense, since the elements $y_J^\alpha$ considered in the RHS are only up to order $r-1$. In addition, for each $1\leq\alpha\leq m$ we have 
\begin{equation*}
\left.\frac{\partial^{|J-1_\mu|}}{\partial x^{J-1_\mu}}\left(b_\beta^\alpha(y)y_{1_\mu}^\beta\right)\right|_x=y_J^\alpha(\mu)+\sum_{I<J-1_\mu}\binom{J-1_\mu}{I}\left.\frac{\partial^{|J-1_\mu-I|}b_\beta^\alpha(y)}{\partial x^{J-1_\mu-I}}\right|_x~y_{I+1_\mu}^\beta={A}_{\mu,J-1_\mu}^\alpha,
\end{equation*}
since $b_\beta^\alpha(y^\alpha=0)=\delta_\beta^\alpha$. Hence, $y_J^\alpha(\mu)$ keeps satisfying \eqref{eq:proof0}.

Let us check that the definition do not depend on the $\mu$ chosen, so we can write $y_J^\alpha=y_J^\alpha(\mu)$. Let $1\leq\nu\leq n$ be another index such that $J_\nu>0$. This enables us to write $J=J'+1_\mu+1_\nu$ for some multi-index $J'$ with $|J'|=r-2$. Hence, $J-1_\mu=J'+1_\nu$ and $J-1_\nu=J'+1_\mu$. Therefore,
\begin{align}\label{eq:proof1}
& \sum_{I<J'+1_\nu}\binom{J'+1_\nu}{I}\left.\frac{\partial^{|J'+1_\nu-I|}b_\beta^\alpha}{\partial x^{J'+1_\nu-I}}(y)\right|_x y_{I+1_\mu}^\beta-\sum_{I<J'+1_\mu}\binom{J'+1_\mu}{I}\left.\frac{\partial^{|J'+1_\mu-I|}b_\beta^\alpha}{\partial x^{J'+1_\mu-I}}(y)\right|_x y_{I+1_\nu}^\beta\\\nonumber
& \hspace{10mm}=\left.\frac{\partial^{|J'+1_\nu|}}{\partial x^{J'+1_\nu}}\left(b_\beta^\alpha(y)y_{1_\mu}^\beta\right)\right|_x-\left.\frac{\partial^{|J'+1_\mu|}}{\partial x^{J'+1_\mu}}\left(b_\beta^\alpha(y)y_{1_\nu}^\beta\right)\right|_x\\\nonumber
& \hspace{10mm}=\frac{\partial^{|J'|}}{\partial (y)x^{J'}}\left(\left.\frac{\partial b_\beta^\alpha}{\partial y^\gamma}\right|_x(y) \, y_{1_\nu}^\gamma y_{1_\mu}^\beta+b_\beta^\alpha(y)y_{1_\mu+1_\nu}^\beta-\left.\frac{\partial b_\beta^\alpha}{\partial y^\gamma}\right|_x(y)\, y_{1_\mu}^\gamma y_{1_\nu}^\beta-b_\beta^\alpha(y)y_{1_\mu+1_\nu}^\beta\right)\\\nonumber
& \hspace{10mm}=\left.\frac{\partial^{|J'|}}{\partial x^{J'}}\left(\left(\frac{\partial b_\gamma^\alpha}{\partial y^\beta}(y)-\frac{\partial b_\beta^\alpha}{\partial y^\gamma}(y)\right)y_{1_\mu}^\gamma y_{1_\nu}^\beta\right)\right|_x
\end{align}
Similarly, by inserting the condition $j_x^{k-1}\sigma_{A}\in\ker j^{k-2}\tilde F$, i.e., $j_x^{k-2}\tilde F^A=j_x^{k-2}\hat 0$, in Lemma \ref{lemma:curvaturecoordinates}, and by using \eqref{eq:proof0}, we obtain
\begin{align}\label{eq:proof2}
{A}_{\mu,J'+1_\nu}^\alpha-{A}_{\nu,J'+1_\mu}^\alpha & =\sum_{I\leq J'}\binom{J'}{I}c_{\lambda\epsilon}^\alpha{A}_{\mu,I}^\lambda{A}_{\nu,J'-I}^\epsilon\\\nonumber
& =\sum_{I\leq J'}\binom{J'}{I}c_{\lambda\epsilon}^\alpha\left.\frac{\partial^{|I|}}{\partial x^I}\left(b_\gamma^\lambda(y)y_{1_\mu}^\gamma\right)\right|_x\left.\frac{\partial^{|J'-I|}}{\partial x^{J'-I}}\left(b_\beta^\epsilon(y)y_{1_\nu}^\beta\right)\right|_x\\\nonumber
& =\left.\frac{\partial^{|J'|}}{\partial x^{J'}}\left(c_{\lambda\epsilon}^\alpha b_\gamma^\lambda(y)b_\beta^\epsilon(y)y_{1_\mu}^\gamma y_{1_\nu}^\beta\right)\right|_x
\end{align}
for each $1\leq\alpha\leq m$, and $1\leq\mu,\nu\leq n$. By gathering \eqref{eq:proof1} and \eqref{eq:proof2}, and by making use of Lemma \ref{lemma:liebracketbs}, we conclude
\begin{align*}
y_J^\alpha(\mu)-y_J^\alpha(\nu) & ={A}_{\mu,J'+1_\nu}^\alpha-{A}_{\nu,J'+1_\mu}^\alpha-\left.\frac{\partial^{|J'|}}{\partial x^{J'}}\left(\left(\frac{\partial b_\gamma^\alpha}{\partial y^\beta}(y)-\frac{\partial b_\beta^\alpha}{\partial y^\gamma}(y)\right)y_{1_\mu}^\gamma y_{1_\nu}^\beta\right)\right|_x\\
& =\left.\frac{\partial^{|J'|}}{\partial x^{J'}}\left(c_{\lambda\epsilon}^\alpha b_\gamma^\lambda(y)b_\beta^\epsilon(y)y_{1_\mu}^\gamma y_{1_\nu}^\beta\right)\right|_x-\left.\frac{\partial^{|J'|}}{\partial x^{J'}}\left(c_{\lambda\epsilon}^\alpha b_\gamma^\lambda(y)b_\beta^\epsilon(y)y_{1_\mu}^\gamma y_{1_\nu}^\beta\right)\right|_x=0.
\end{align*}

Once we have defined $j_x^k s=(x^\mu,y^\alpha;y_J^\alpha)\in J^k P$, the last step is to check that $\Theta_k(j_x^k s)=j_x^{k-1}\sigma_{A}$, which is straightforward from \eqref{eq:proof0} and Proposition \ref{prop:Thetakcoordinates}.

Lastly, we show that $\Xi_k$ is proper. Let $\left([j_{x_n}^k s_n]_G\right)_{n=1}^\infty$ be a sequence in $\left.\left(J^k P\right)\right/G$ such that $\left(j_{x_n}^{k-1}\sigma_{{A}_n }=\Xi_k\left([j_{x_n}^k s_n]_G\right)\right)_{n=1}^\infty$ converges to some $j_x^{k-1}\sigma_{A}\in J^{k-1}C(P)$. Since $j_{x_n}^{k-2}\tilde F_{{A}_n}=j_{x_n}^{k-2}\hat 0$ for each $n\in\mathbb N$, by continuity we conclude that $j_x^{k-2}\tilde F^A=j_x^{k-2}\hat 0$, that is, $j_x^{k-1}\sigma_{A}\in\ker j^{k-2}\tilde F=\operatorname{im}\Xi_k$. Thus, there exists $j_x^k s\in J^k P$ such that $\Xi_k([j_x^k s]_G)=j_x^{k-1}\sigma_{A}$. In coordinates, we write $j_{x_n}^k s_n=(x_n^\mu,y_n^\alpha;(y_n)_J^\alpha)$, $j_{x_n}^{k-1}\sigma_{{A}_n}=(x_n^\mu,({A}_n)_\mu^\alpha;({A}_n)_{\mu,J}^\alpha)$, $j_x^k s=(x^\mu,y^\alpha;y_J^\alpha)$ and $j_x^{k-1}\sigma_{A}=(x^\mu,{A}_\mu^\alpha;{A}_{\mu,J}^\alpha)$. Without loss of generality, we may assume that $y_n^\alpha=0$ and $y^\alpha=0$, $1\leq\alpha\leq m$. This way, by recalling Proposition \ref{prop:Thetakcoordinates}, we have
\begin{equation*}
\lim_n\left.\frac{\partial^{|J|}}{\partial x^J}\left(b_\beta^\alpha(y_n)(y_n)_{1_\mu}^\beta\right)\right|_{x_n}=\lim_n({A}_n)_{\mu,J}^\alpha={A}_{\mu,J}^\alpha=\left.\frac{\partial^{|J|}}{\partial x^J}\left(b_\beta^\alpha(y)y_{1_\mu}^\beta\right)\right|_x
\end{equation*}
for each $0\leq|J|\leq k-1$, $1\leq\alpha\leq m$ and $1\leq\mu\leq n$. Subsequently, by recurrence as above, it can be checked that $\lim_n(y_n)_J^\alpha=y_J^\alpha$ for each $0\leq|J|\leq k$ and $1\leq\alpha\leq m$ and, thus, $\left(j_{x_n}^k s_n\right)_{n=1}^\infty$ converges to $j_x^k s$.
\end{proof}

By gathering the previous results, we obtain the geometry of the reduced configuration bundle. 

\begin{corollary}\label{corollary:isomorphismquotient}
For each $k\geq 2$, there is an isomorphism of fiber bundles covering the identity $\operatorname{id}_X$,
\begin{equation*}
\left(J^k P\right)/ G\simeq \ker j^{k-2}\tilde F=\left\{j_x^{k-1}\sigma_{A}\in J^{k-1}C(P)\mid j_x^{k-2}\tilde F^A=j_x^{k-2}\hat 0\right\}.
\end{equation*}
\end{corollary}

\begin{proof}
Let us show that $\Xi_k$ is an immersion, i.e., that $d\Xi_k$ is everywhere injective. Since the property is local we can work in a trivialization of $\pi_{P,X}$, that is, we assume that $P=X\times G$. Given two sections $s=(\operatorname{id}_X,f),s'=(\operatorname{id}_X,f')\in\Gamma(\pi_{P,X})$, we define $\tau:P\to P$ as $\tau(x,g)=\left(x,f'(x)f(x)^{-1}g\right)$ for each $x\in X$, which is a principal bundle automorphism such that $\tau\circ s=s'$. Besides, the jet extension $j^1\tau:J^1 P\to J^1 P$ is $G$-equivariant, i.e., $j^1\tau(j^1_x s)\cdot g=j^1\tau\left((j_x^1 s)\cdot g\right)$ for each $g\in G$ and each $j_x^1 s\in J^1 P$. Hence, it induces an automorphism of the bundle of connections, $\hat\tau=[j^1\tau]_G: C(P)\to C(P)$. By taking the jet lift of these maps, we arrive at the following commutative diagram
\begin{equation*}
\begin{tikzpicture}
\matrix (m) [matrix of math nodes,row sep=3em,column sep=5em,minimum width=2em]
{	J^k P & J^{k-1}C(P)\\
	J^k P & J^{k-1}C(P)\\};
\path[-stealth]
(m-1-1) edge [] node [above] {$\Theta_k$} (m-1-2)
(m-1-1) edge [] node [left] {$j^k\tau$} (m-2-1)
(m-1-2) edge [] node [right] {$j^{k-1}\hat\tau$} (m-2-2)
(m-2-1) edge [] node [above] {$~\Theta_k$} (m-2-2);
\end{tikzpicture} 
\end{equation*}
Therefore, $\Theta_k$ has constant rank, since $\tau$ and $\hat\tau$ are isomorphisms and this construction can be done for every pair of sections $s,s'\in\Gamma(\pi_{P,X})$. By using this and the fact that $\pi_k$ is a submersion, and by recalling \eqref{eq:diagram}, we conclude that $\Xi_k$ has constant rank too. Since $\Xi_k$ is injective (recall Proposition \ref{prop:Psikproperties}\emph{(i)}), the Global Rank Theorem (see, for example, \cite[Theorem 4.14]{Le2012}) ensures that $\Xi_k$ is an immersion.

In short, $\Xi_k$ is an injective immersion and proper (by \emph{(ii)} of Theorem \ref{theorem:imageandproper}), whence it is an embedding (cf. \cite[Proposition 4.22]{Le2012}). Hence, $\operatorname{im}\Xi_k=\ker j^{k-2}\tilde F$ is a submanifold of $J^{k-1}C(P)$ and, again by the Global Rank Theorem, $\Xi_k:\left.\left(J^k P\right)\right/G\to\ker j^{k-2}\tilde F$ is a diffeomorphism.
\end{proof}

\noindent {\bf Notation.} Henceforth, we will denote
\[
C^{k-1}_0(P)=\operatorname{im}\Xi_k=\ker j^{k-2}\tilde F.
\]

%%%%%%%%%%%%%%%%%%%%%%%%%%%%%%%%%%%%%%%%%%%%%%%%%%%%%%%
\section{Higher-order Euler--Poincar\'e reduction}\label{sec:redeqs}

At this point, we are ready to develop the Euler--Poincar\'e reduction for higher-order field theories. Given a principal $G$-bundle, $\pi_{P,X}:P\to X$, over a boundaryless manifold $X$ and $k\geq 2$, we consider a $k$-th order Lagrangian density, $\mathfrak L=Lv:J^k P\to\bigwedge^n T^*X$, that is $G$-invariant, i.e.,
\begin{equation*}
L\left(R_g^{(k)}\left(j_x^k s\right)\right)=L\left(j_x^k s\right),\qquad j_x^k s\in J^k P,~g\in G.
\end{equation*}
As a result, we have a \emph{dropped} or \emph{reduced Lagrangian}, 
\begin{equation*}
\mathfrak l=lv:\left.\left(J^k P\right)\right/G\to\textstyle\bigwedge^n T^*X,\qquad\left[j_x^k s\right]_G\mapsto \mathfrak l\left(\left[j_x^k s\right]_G\right)=\mathfrak L\left(j_x^k s\right).
\end{equation*}
By means of Corollary \ref{corollary:isomorphismquotient}, we may regard the reduced Lagrangian as defined on $C^{k-1}_0(P)=\ker j^{k-2}\tilde F\subset J^{k-1}C(P)$. For an open subset $\mathcal U\subset X$ with compact closure and $s\in\Gamma\left(\mathcal U,\pi_{P,X}\right)$, the corresponding \emph{reduced section} is defined as
\begin{equation*}
\sigma_{A}=\left[j^1 s\right]_G\in\Gamma\left(\mathcal U,\pi_{C(P),X}\right).
\end{equation*}
A variation $\delta s$ of the original section induces a variation $\delta\sigma_{A}$ of the reduced section. By construction, the original and the reduced variations satisfy $\mathfrak L\left(j^k s_t\right)=\mathfrak l\left(j^{k-1}\left(\overline\sigma_{A}\right)_t\right)$ for all $t\in(-\epsilon,\epsilon)$ and, thus,
\begin{equation}\label{eq:equivalenciavariaciones}
\left.\frac{d}{dt}\right|_{t=0}\int_{\mathcal U}\mathfrak L\left(j^k s_t\right)=\left.\frac{d}{dt}\right|_{t=0}\int_{\mathcal U}\mathfrak l\left(j^{k-1}\left(\sigma _A\right)_t\right).
\end{equation}

If we choose a fixed principal connection ${A}_0\in\Omega^1(P,\mathfrak g)$ on $\pi_{P,X}$, we can identify the bundle of connections with the modelling vector bundle,
\begin{equation}\label{eq:isomorphismsCP}
C(P)\ni\sigma_{A}(x)\overset{1:1}{\longleftrightarrow}\overline\sigma_{A}(x)=\sigma_{A}(x)-\sigma_{{A}_0}(x)\in T^*X\otimes \ad(P).
\end{equation}
By taking the $(k-1)$-jet lift of this map, we can also identify the reduced configuration space with some subbundle of $J^{k-1}\left(T^*X\otimes \ad(P)\right)$, that is,
\begin{equation*}
J^{k-1}C(P)\supset C^{k-1}_0(P)\overset{\sim}{\longleftrightarrow}\overline{C^{k-1}_0(P)}\subset J^{k-1}\left(T^*X\otimes \ad(P)\right),
\end{equation*}
so that the reduced Lagrangian could be regarded as
\begin{equation*}
\mathfrak l:\overline{C^{k-1}_0(P)}\to\textstyle\bigwedge^n T^*X.
\end{equation*} 
The whole situation is summarized in the following commutative diagram:
\begin{equation*}
\begin{tikzpicture}
\matrix (m) [matrix of math nodes,row sep=3em,column sep=3em,minimum width=2em]
{	J^k P & & &\\
	 & \left.\left(J^k P\right)\right/G & C^{k-1}_0(P) & \overline{C^{k-1}_0(P)}\\
	 J^1 P & & &\\
	 P & J^1 P/G & C(P) & T^*X\otimes \ad(P)\\
	 \mathcal U\left(\subset X\right) & & &\\};
\path[-stealth]
(m-1-1) edge [] node [] {} (m-3-1)
(m-3-1) edge [] node [] {} (m-4-1)
(m-4-1) edge [] node [] {} (m-5-1)
(m-5-1) edge [bend left=35,dashed] node [left] {$s$} (m-4-1)
(m-5-1) edge [bend left=50,dashed] node [left] {$j^k s$} (m-1-1)
(m-1-1) edge [] node [above] {$\pi_k$} (m-2-2)
(m-3-1) edge [] node [above] {$\pi_1$} (m-4-2)
(m-4-2) edge [] node [] {} (m-5-1)
(m-2-2) edge [] node [above] {$\Xi_k$} (m-2-3)
(m-2-3) edge [] node [above] {$\sim$} (m-2-4)
(m-2-4) edge [] node [] {} (m-4-4)
(m-5-1) edge [bend right,dashed] node [above] {$\sigma_{A}$} (m-4-3)
(m-5-1) edge [bend right,dashed] node [above] {$\overline\sigma_{A}$} (m-4-4)
(m-2-2) edge [] node [] {} (m-2-3)
(m-2-2) edge [] node [] {} (m-4-2)
(m-2-3) edge [] node [] {} (m-4-3)
%(m-5-1) edge [bend right,dashed] node [above] {$\overline s$} (m-2-3)
(m-4-2) edge [] node [above] {$\sim$} (m-4-3)
(m-4-3) edge [] node [above] {$\sim$} (m-4-4);
\end{tikzpicture} 
\end{equation*}

\begin{definition}
Let $\nabla^{{A}_0}:\Gamma\left(\pi_{\ad(P),X}\right)\to\Gamma\left(\pi_{T^*X\otimes \ad(P),X}\right)$ be the linear connection on the adjoint bundle, $\pi_{\ad(P),X}$, induced by ${A}_0$. The \emph{divergence of $\nabla^{{A}_0}$} is minus the adjoint of $\nabla^{{A}_0}$, i.e., the map $\textrm{\normalfont div}^{{A}_0}:\Gamma\left(\pi_{TX\otimes \ad(P)^*,X}\right)\to\Gamma\left(\pi_{\ad(P)^*,X}\right)$ implicitly defined by
\begin{equation*}
\int_X\left\langle\eta,\nabla^{{A}_0}\xi\right\rangle v=-\int_X\left\langle\textrm{\normalfont div}^{{A}_0}\eta,\xi\right\rangle v,\qquad\xi\in\Gamma\left(\pi_{\ad(P),X}\right),~\eta\in\Gamma\left(\pi_{TX\otimes \ad(P)^*,X}\right).
\end{equation*}
\end{definition}

Given $\overline\sigma_{A}\in\Gamma\left(\pi_{T^*X\otimes \ad(P),X}\right)$, we define the map $\ad_{\overline\sigma_{A}}^*:\Gamma\left(\pi_{TX\otimes \ad(P)^*,X}\right)\to\Gamma\left(\pi_{\ad(P)^*,X}\right)$ as the coadjoint representation of $\ad(P)^*$ and the dual pairing of $TX$ and $T^*X$. In the following, since $\pi_{T^*X\otimes \ad(P),X}$ is a vector bundle, we identify its vertical bundle with itself, $V\left(T^*X\otimes \ad(P)\right)\simeq T^*X\otimes \ad(P)$, and analogous for its dual.

\begin{theorem}[Higher-order Euler--Poincaré field equations]\label{theorem:reducedequations}
Let $k\geq 2$, $\pi_{P,X}:P\to X$ be a principal $G$-bundle and $\mathfrak L:J^k P\to\bigwedge^n T^*X$ be a $k$-th order, $G$-invariant Lagrangian density. Let $\mathfrak l:C^{k-1}_0(P)\to\bigwedge^n T^*X$ be the reduced Lagrangian density, where $C^{k-1}_0(P)\simeq\left.\left(J^k P\right)\right/G$ is the reduced space, and consider an extension $\hat{\mathfrak l}:J^{k-1}\left(C(P)\right)\to\bigwedge^n T^*X$ of the reduced Lagrangian $\mathfrak l$ to the whole jet.
Let ${A}_0\in\Omega^1(P,\mathfrak g)$ be a principal connection on $\pi_{P,X}$, which enables us to identify $C(P)\simeq T^*X\otimes \ad(P)$, and let $\nabla^{{A}_0}$ be the linear connection on $\pi_{\ad(P),X}$ induced by ${A}_0$. 

Given a (local) section $s\in\Gamma(\mathcal U,\pi_{P,X})$, where $\mathcal U\subset X$ is an open subset with compact closure, and the corresponding reduced section $\sigma_{A}\in\Gamma\left(\mathcal U,\pi_{C(P),X}\right)$, the following statements are equivalent:
\begin{enumerate}[(i)]
    \item The variational principle $\delta\displaystyle\int_{\mathcal U}\mathfrak L\left(j^k s\right)=0$ holds for arbitrary variations of $s$ such that $j_x^{k-1}\,\delta s=0$ for each $x\in\partial\mathcal U$.
    \item The section $s$ satisfies the Euler--Lagrange field equations, $\left(j^{2k}s\right)^*\mathcal{EL}(\mathfrak L)=0$.
    \item The variational principle $\delta\displaystyle\int_{\mathcal U}\mathfrak l\left(j^{k-1}\sigma_{A}\right)=0$ holds for variations of the form 
    \begin{equation*}
    \delta\sigma_{A}=\nabla^{{A}_0}\xi-\left[\overline\sigma_{A},\xi\right]\in\Gamma\left(\mathcal U,\pi_{T^*X\otimes \ad(P),X}\right),
    \end{equation*}
    where $\xi\in\Gamma(\mathcal U,\pi_{\ad(P),X})$ is an arbitrary section such that $j_x^{k-1}r\xi=0$ for each $x\in\partial\mathcal U$.
    \item The reduced section $\overline\sigma_{A}$ satisfies the \emph{$k$-th order Euler--Poincaré field equations}:
    \begin{equation}
    \label{HOEP}
    \left(\textrm{\normalfont div}^{{A}_0}-\ad_{\overline\sigma_{A}}^*\right)\left(\left(j^{2k-2}\sigma_{A}\right)^*\mathcal{EL}\left(\hat{\mathfrak l}\right)\right)=0,
    \end{equation}
    where $\mathcal{EL}\left(\hat{\mathfrak l}\right)\in\Omega^n\left(J^{2k-2}(C(P)),TX\otimes \ad(P)^*\right)$ is the Euler--Lagrange form of $\hat{\mathfrak l}$.
\end{enumerate}
\end{theorem}

\begin{proof}
The equivalence between $(i)$ and $(ii)$ is established in Theorem \ref{theorem:ELequations}. In addition, the induced variations of the reduced section were computed in \cite[Theorem 1]{CaGaRa2001}. Hence, the equivalence between $(i)$ and $(iii)$ is straightforward from this fact and equation \eqref{eq:equivalenciavariaciones}. To conclude, let us show the equivalence between $(iii)$ and $(iv)$. Firstly, note that given a variation of the original section, together with the induced variation of the reduced section, we have $j_x^{k-1}\left(\sigma_{A}\right)_t\in C^{k-1}_0(P)$ for all $t\in(-\epsilon,\epsilon)$ and $x\in\mathcal U$. Hence, $\delta\left(j^{k-1}\sigma_{A}\right)(x)=\left(\delta\sigma_{A}\right)^{(k-1)}(x)\in T_{j_x^{k-1}\sigma_{A}}\left(C^{k-1}_0(P)\right)$. By applying the explicit expression of the reduced variations and Proposition \ref{prop:actionvariation}, we obtain
\begin{align*}
\delta\int_{\mathcal U}\hat{\mathfrak l}\left(j^{k-1}\sigma_{A}\right) & =\int_{\mathcal U}\left\langle\left(j^{2k-2}\sigma_{A}\right)^*\mathcal{EL}\left(\hat{\mathfrak l}\right),\delta\sigma_{A}\right\rangle\\
& =\int_{\mathcal U}\left\langle\left(j^{2k-2}\sigma_{A}\right)^*\mathcal{EL}\left(\hat{\mathfrak l}\right),\nabla^{{A}_0}\xi-\left[\overline\sigma_{A},\xi\right]\right\rangle\\
& =-\int_{\mathcal U}\left\langle\textrm{\normalfont div}^{{A}_0}\left(\left(j^{2k-2}\sigma_{A}\right)^*\mathcal{EL}\left(\hat{\mathfrak l}\right)\right)-\ad_{\overline\sigma_{A}}^*\left(\left(j^{2k-2}\sigma_{A}\right)^*\mathcal{EL}\left(\hat{\mathfrak l}\right)\right),\xi\right\rangle=0.
\end{align*}
Since this holds for every $\xi\in\Gamma\left(\mathcal U,\pi_{\ad(P),X}\right)$ vanishing at the boundary, we conclude.
\end{proof}

\begin{remark}
Since (the $(k-1)$-th jet lift of) the reduced section lies in $C_0^{k-1}(P)$ and the infinitesimal reduced variations are tangent to this space, the solutions of the reduced equations do not depend on the choice of the extended Lagrangian $\hat{\mathfrak l}:J^{k-1}C(P)\to\bigwedge^n T^*X$. That is, if we have two such extensions $\hat{\mathfrak l}_1,\hat{\mathfrak l}_2$ with $\hat{\mathfrak l}_1|_{C^{k-1}_0(P)}=\hat{\mathfrak l}_2|_{C^{k-1}_0(P)}$, then a reduced section is critical with respect for $\hat{\mathfrak l}_1$ if and only if it is critical for $\hat{\mathfrak l}_2$. In other words, $C_0^{k-1}(P)\subset J^{k-1}C(P)$ may be regarded as a kind of holonomic constraint.
\end{remark}

\begin{remark}
    If we choose the connection $A_0 = A$ itself, that is, the connection defined by the reduced section $\sigma _A =[j^1s]_G$, the reduced equation \eqref{HOEP} has the simpler form
    \[
    \textrm{\normalfont div}^{A}\left(\left(j^{2k-2}\sigma_{A}\right)^*\mathcal{EL}\left(\hat{\mathfrak l}\right)\right)=0.
    \]
\end{remark}
Let us find the local expression of the reduced equations. As in Section \ref{sec:reducedspace}, let $\{B_\alpha\mid 1\leq\alpha\leq m\}$ be a basis of $\mathfrak g$, $(x^\mu,y^\alpha)$ be normal bundle coordinates for $\pi_{Y,X}$ and $(x^\mu,{A}_\mu^\alpha)$ be bundle coordinates for $\pi_{C(P),X}$. Similarly, consider the adapted coordinates $(x^\mu,y^\alpha;y_J^\alpha)$ and $(x^\mu,{A}_\mu^\alpha;{A}_{\mu,J}^\alpha)$ for $J^k Y$ and $J^{k-1}C(P)$, respectively. For the sake of simplicity, we fix the flat connection on $\pi_{Y,X}$ under these coordinates, i.e., ${A}_0=dy^\alpha\otimes B_\alpha$. This way, $(x^\mu,{A}_\mu^\alpha)$ may also be regarded as bundle coordinates on $\pi_{T^*X\otimes \ad(P),X}$, and we denote by $(x^\mu,v_\alpha^\mu)$ the bundle coordinates on the dual bundle, $\pi_{TX\otimes \ad(P)^*,X}$.

\begin{lemma}\label{lemma:divergencecoadjoint}
With the above bundle coordinates, let $\eta=\eta_\alpha^\mu\,\partial_\mu\otimes B^\alpha\in\Gamma\left(\pi_{TX\otimes \ad(P)^*,X}\right)$ and $\overline\sigma_{A}=\sigma_\mu^\alpha\,dx^\mu\otimes B_\alpha\in\Gamma\left(\pi_{T^*X\otimes \ad(P),X}\right)$. Then we have 
\begin{equation*}
\textrm{\normalfont div}^{{A}_0}(\eta)=\partial_\mu\eta_\alpha^\mu\,B^\alpha,\qquad \ad_{\overline\sigma_{A}}^*(\eta)=-c_{\beta\gamma}^\alpha\,\eta_\alpha^\mu\,\sigma_\mu^\beta\,B^\gamma.
\end{equation*}
\end{lemma}

\begin{proof}
Let $\xi=\xi^\alpha B_\alpha\in\Gamma\left(\pi_{\ad(P),X}\right)$ and note that $\nabla^{{A}_0}\xi=\partial_\mu\xi^\alpha dx^\mu\otimes B_\alpha$. By definition of divergence and the integration by parts formula \cite[Lemma 4.5]{CaLeMa2010}, we finish
\begin{align*}
\int_X\left\langle\textrm{\normalfont div}^{{A}_0}(\eta),\xi\right\rangle v & =-\int_X\left\langle\eta,\nabla^{{A}_0}\xi\right\rangle v\\
& =-\int_X\left\langle\eta_\alpha^\mu\partial_\mu\otimes B^\alpha,\partial_\mu\xi^\alpha dx^\mu\otimes B_\alpha\right\rangle v\\
& =-\int_X\left(\eta_\alpha^\mu\partial_\mu\xi^\alpha\right)v=\int_X\left(\partial_\mu\eta_\alpha^\mu\xi^\alpha\right)v.
\end{align*}
For the second part, we have
\begin{align*}
\left\langle \ad_{\overline\sigma_{A}}^*(\eta),\xi\right\rangle & =-\left\langle \eta,\ad_{\overline\sigma_{A}}(\xi)\right\rangle=-\left\langle \eta_\alpha^\mu\,\partial_\mu\otimes B^\alpha,\left[\sigma_\mu^\alpha\,dx^\mu\otimes B_\alpha,\xi^\alpha B_\alpha\right]\right\rangle\\
& =-\left\langle \eta_\alpha^\mu\,\partial_\mu\otimes B^\alpha,c_{\beta\gamma}^\alpha\,\sigma_\mu^\beta\,\xi^\gamma\,dx^\mu\otimes B_\alpha\right\rangle=-c_{\beta\gamma}^\alpha\,\eta_\alpha^\mu\,\sigma_\mu^\beta\,\xi^\gamma.
\end{align*}
\end{proof}

\begin{corollary}[Local equations]\label{corollary:localequations}
With the above coordinates, we write
\begin{equation*}
\overline\sigma_{A}=\sigma_\mu^\alpha dx^\mu\otimes B_\alpha\in\Gamma\left(\mathcal U,\pi_{T^*X\otimes \ad(P),X}\right)    
\end{equation*}
for the reduced section. Then the local expression of the $k$-th order Euler--Poincaré field equations is given by
\begin{equation*}
\sum_{|J|=0}^{k-1}(-1)^{|J|}\left(\frac{\partial^{|J+1_\mu|}}{\partial x^{J+1_\mu}}\left(\frac{\partial\hat l}{\partial{A}_{\mu,J}^\alpha}\left(j^{k-1}\overline\sigma_{A}\right)\right)-c_{\beta\alpha}^\gamma\sigma_\mu^\beta\frac{\partial^{|J|}}{\partial x^J}\left(\frac{\partial\hat l}{\partial{A}_{\mu,J}^\gamma}\left(j^{k-1}\overline\sigma_{A}\right)\right)\right)=0,\qquad1\leq\alpha\leq m.
\end{equation*}
\end{corollary}

\begin{proof}
From \eqref{eq:ELformlocal} and \eqref{eq:totalderivativesection}, it is clear that
\begin{equation*}
\left(j^{2k-2}\overline\sigma_{A}\right)^*\mathcal{EL}\left(\hat{\mathfrak l}\right)=\sum_{|J|=0}^{k-1}(-1)^{|J|}\frac{\partial^{|J|}}{\partial x^J}\left(\frac{\partial\hat l}{\partial{A}_{\mu,J}^\alpha}\left(j^{k-1}\overline\sigma_{A}\right)\right)v\otimes d{A}_\mu^\alpha.
\end{equation*}
Recall that we are identifying $V^*\left(T^*X\otimes \ad(P)\right)\simeq TX\otimes \ad(P)^*$, i.e., $d{A}_\mu^\alpha\simeq\partial_\mu\otimes B^\alpha$. The result is now an straightforward consequence of $(iv)$ of Theorem \ref{theorem:reducedequations} and Lemma \ref{lemma:divergencecoadjoint}.
\end{proof}

At last, note that a solution of the reduced equations lying in the reduced space corresponds to a flat principal connection and, thus, the reconstruction condition is automatically satisfied.

\begin{theorem}[Reconstruction]\label{theorem:reconstruction}
In the conditions of Theorem \ref{theorem:reducedequations}, let $\sigma_{A}\in\Gamma\left(\mathcal U,\pi_{C(P),X}\right)$ be a solution of the $k$-th order Euler--Poincaré field equations \eqref{HOEP} such that $j^{k-1}\sigma_{A}\in\Gamma\left(\mathcal U,\pi_{C^{k-1}_0(P),X}\right)$ and ${A}\in\Omega^1(P|_\mathcal{U},\mathfrak g)$ has trivial holonomy on a domain. Then there exists a section $s\in\Gamma\left(\mathcal U,\pi_{P,X}\right)$ that is critial for the original variational problem defined by $\mathfrak L$ and such that $\sigma_{A}=\left[j^1 s\right]_G$. Furthermore, any critical section on $\mathcal{U}$ of $\mathfrak L$ is of the form $s\cdot g$ for certain $g\in G$.
\end{theorem}

\begin{proof}
From the condition $j^{k-1}\sigma_{A}\in\Gamma\left(\mathcal U,\pi_{C^{k-1}_0(P),X}\right)$ we know that ${A}\in\Omega^1(P|_{\mathcal U},\mathfrak g)$ is a flat connection. Therefore, there exists a foliation of $P|_{\mathcal U}$ given by the integral leaves of ${A}$. The trivial holonomy ensures that each integral leaf intersects once and only once each fiber of $P|_{\mathcal U}$. Subsequently, given a integral leaf of ${A}$, it defines a section $s\in\Gamma(\mathcal U,\pi_{P,X})$ projecting onto $\sigma_{A}$. This section is critical thanks to theorem \ref{theorem:reducedequations}. Moreover, the remaining integral leaves of ${A}$ are obtained from $s$ as follows, 
\begin{equation*}
\mathcal F_g=s(\mathcal U)\cdot g\subset P|_{\mathcal U},\qquad g\in G.
\end{equation*}
Hence, given another critical section, $\tilde s\in\Gamma(\mathcal U,\pi_{P,X})$, projecting onto $\sigma_{A}$, we conclude that $\tilde s=s\cdot g$ for some $g\in G$, since $\tilde s(\mathcal U)$ must belong to the previous family.
\end{proof}

If $X$ is a simply connected manifold, then every flat connection has trivial holonomy and we have the following equivalence
\begin{equation*}
\left(j^{2k}s\right)^*\mathcal{EL}(\mathfrak L)=0\qquad\Longleftrightarrow\qquad\left\{\begin{array}{l}
\displaystyle\left(\textrm{\normalfont div}^{{A}_0}-\ad_{\overline\sigma_{A}}^*\right)\left(\left(j^{2k-2}\overline\sigma_{A}\right)^*\mathcal{EL}\left(\hat{\mathfrak l}\right)\right)=0,\\
\displaystyle\tilde F^A=0.
\end{array}\right.
\end{equation*}
In an arbitrary manifold the equivalence above is valid locally only.

\begin{remark}
    That fact that the reconstruction process requires the flatness of a connection is a characteristic trait in Field Theories that does not show up in Mechanics, i.e., when $\dim X=1$ (see for example \cite{CaGaRa2001,ElGaHoRa2011} for comments on this situation). However, there is a significant difference from first order to higher-order reduction. Indeed, the flatness condition is a compatibility equation that must be incorporated by hand for $k=1$ theories, whereas this condition comes directly from the geometry of the reduced phase space for $k>1$. In other words, for higher-order Euler--Poincaré reduction the reconstruction is not inserted ad hoc, but it is intrinsic in the nature of the reduced sections.
\end{remark}

\section{Noether theorem}\label{sec:noether}

The well-known Noether theorem establishes that infinitesimal symmetries of a Lagrangian density lead to conserved quantities for the dynamics of the system. As in the previous section, let $\mathfrak L=Lv:J^k P\to\bigwedge^n T^*X$ be a $G$-invariant $k$-th order Lagrangian on a principal $G$-bundle, $\pi_{P,X}:P\to X$. The aim of this section is to prove that the conservation laws of $\mathfrak L$ arising from its $G$-invariance are equivalent to the higher-order Euler--Poincaré field equations. 

We start by recalling the definition of an infinitesimal symmetry and the Noether theorem for higher-order Lagrangian densities. (cf. \cite[\S 10]{Ma1985}).

\begin{definition}
An \emph{infinitesimal symmetry} of $\mathfrak L$ is a projectable vector field $U\in\mathfrak X(P)$ such that $\pounds_{U^{(k)}}\mathfrak L=0$, where $\pounds$ denotes the Lie derivative and $U^{(k)}\in\mathfrak X(J^k P)$ is the $k$-th order prolongation of $U$.
\end{definition}

For a vertical vector field $V\in\mathfrak X(P)$, being an infinitesimal symmetry is equivalent to $\pounds_{V^{(2k-1)}}\Theta_{\mathfrak L}=0$, for any covariant Cartan form, $\Theta_{\mathfrak L}\in\Omega^n(J^{2k-1} P)$, of $\mathfrak L$.

\begin{theorem}[Noether theorem]
Let $s\in\Gamma(\pi_{P,X})$ and $U\in\mathfrak X(U)$ be a critical section and an infinitesimal symmetry of $\mathfrak L$, respectively. Then
\begin{equation*}
{\rm d}\left(\left(j^{2k-1}s\right)^*\iota_{U^{(2k-1)}}\Theta_{\mathfrak L}\right)=0,
\end{equation*}
for any covariant Cartan form, $\Theta_{\mathfrak L}$, of $\mathfrak L$.
\end{theorem}

Henceforth, we focus on the covariant Cartan form given locally by \eqref{eq:Cartanformlocal}. As our Lagrangian is $G$-invariant, each infinitesimal generator $\xi^*\in\mathfrak X(Y)$ of an element $\xi\in\mathfrak g$ is an infinitesimal symmetry of $\mathfrak L$. We define the \emph{Noether current} as the $\mathfrak g^*$-valued form $\mathcal J\in\Omega^{n-1}\left(J^{2k-1}P,\mathfrak g^*\right)$ given by
\begin{equation*}
\left\langle\mathcal J\left(j_x^{2k-1}s\right),\xi\right\rangle=\iota_{(\xi^*)^{(2k-1)}}\Theta_{\mathfrak L}\left(j_x^{2k-1}s\right),\qquad j_x^{2k-1}s\in J^{2k-1}P,~\xi\in\mathfrak g,
\end{equation*}
where $\langle\cdot,\cdot\rangle$ denotes the dual pairing, as usual. Observe that the Noether theorem ensures that
\begin{equation}\label{eq:noether}
{\rm d}\left( \left(j^{2k-1}s\right)^*\mathcal J\right)=0
\end{equation}
for every critical section $s\in\Gamma(\pi_{Y,X})$. We are ready to show the main theorem of this section.

\begin{theorem}
Let $\pi_{P,X}:P\to X$ be a principal $G$-bundle, $\mathfrak L=Lv:J^k P\to\bigwedge^n T^*X$ be a $G$-invariant Lagrangian, $\mathfrak l=lv:C^{k-1}_0(P)\to\bigwedge^n T^*X$ be the reduced Lagrangian, and $\mathcal J\in\Omega^{n-1}\left(J^{2k-1}P,\mathfrak g^*\right)$ be the Noether current given by the Cartan form $\Theta_{\mathfrak L}\in\Omega^n(J^{2k-1} P)$ locally defined by \eqref{eq:Cartanformlocal}. Consider a section, $s\in\Gamma(\pi_{Y,X})$ and the corresponding reduced section, $\sigma_{A}\in\Gamma\left(\pi_{C(P),X}\right)$. Then the Noether equation \eqref{eq:noether} holds if and only if the $k$-th order Euler--Poincaré field equations hold for the principal connection ${A}\in\Omega^1(P,\mathfrak g)$, i.e.,
\begin{equation*}
\textrm{\normalfont div}^{A}\left(\left(j^{2k-2}\sigma _A\right)^*\mathcal{EL}\left(\hat l\right)\right)=0,
\end{equation*}
where $\hat l:J^{k-1}(C(P))\to\bigwedge^n T^*X$ is an extension of the reduced Lagrangian.
\end{theorem}

\begin{proof}
Let $(x^\mu,y^\alpha;y_J^\alpha)$ be adapted coordinates for $J^{2k}P$. Some conditions will be imposed on these coordinates along the proof. Firstly, suppose that $(y^\alpha)$ are normal coordinates on some neighborhood of the identity element. Hence, given a basis $\{B_\alpha\mid1\leq\alpha\leq m\}$ of $\mathfrak g$, we may use equation \eqref{eq:xitopartiallocal} and Proposition \ref{prop:jetprolongationvector} (recall that $\xi^*\in\mathfrak X(P)$ is a vertical vector field) to obtain
\begin{equation*}
(\xi^*)_y^{(2k-1)}=\xi^\beta a_\beta^\alpha(y)\partial_\alpha+\xi^\beta\frac{d^{|J|}a_\beta^\alpha(y)}{dx^J}\partial_\alpha^J,\qquad y=(x^\mu,y^\alpha)\in P.
\end{equation*}
If the coordinates are chosen so that the volume form is given by $v=d^n x=dx^1\wedge\dots\wedge dx^n$, the left interior product of the covariant Cartan form \eqref{eq:Cartanformlocal} by $(\xi^*)^{(2k-1)}$ reads
\begin{equation*}
\iota_{(\xi^*)^{(2k-1)}}\Theta_{\mathfrak L}=\sum_{|I|=0}^{k-1}\sum_{|J|=0}^{k-|I|-1}(-1)^{|J|} \xi^\beta\frac{d^{|J|}}{dx^J}\left(\frac{\partial L}{\partial y_{I+J+1_\mu}^\alpha}\right)\frac{d^{|I|}a_\beta^\alpha(y)}{dx^I}d^{n-1}x_\mu,
\end{equation*}
where $d^{n-1}x_\mu=\iota_{\partial_\mu}d^n x$. In coordinates, the critical section is written as $s(x^\mu)=(x^\mu,s^\alpha(x^\mu))$, for some local functions $s^\alpha\in C^\infty(X)$, $1\leq\alpha\leq m$. From the previous expression and \eqref{eq:totalderivativesection}, we get
\begin{equation}\label{eq:proofnoether0}
 \left(j^{2k-1}s\right)^*\mathcal J=\sum_{|I|=0}^{k-1}\sum_{|J|=0}^{k-|I|-1}(-1)^{|J|} \xi^\beta\frac{\partial^{|J|}}{\partial x^J}\left(\frac{\partial L}{\partial y_{I+J+1_\mu}^\alpha}\left(j^k s\right)\right)\frac{\partial^{|I|}a_\beta^\alpha(s^\alpha)}{\partial x^I}d^{n-1}x_\mu.
\end{equation}
%Hence, 
%\begin{multline*}
%{\rm d}\left( \left(j^{2k-1}s\right)^*\mathcal J\right)=\sum_{|I|=0}^{k-1}\sum_{|J|=0}^{k-|I|-1}(-1)^{|J|} \xi^\beta\left(\frac{\partial^{|J+1_\mu|}}{\partial x^{J+1_\mu}}\left(\frac{\partial L}{\partial y_{I+J+1_\mu}^\alpha}\left(j^k s\right)\right)\frac{\partial^{|I|}a_\beta^\alpha(s^\alpha)}{\partial x^I}\right.\\
%\left.+\frac{\partial^{|J|}}{\partial x^J}\left(\frac{\partial L}{\partial y_{I+J+1_\mu}^\alpha}\left(j^k s\right)\right)\frac{\partial^{|I+1_\mu|}a_\beta^\alpha(s^\alpha)}{\partial x^{I+1_\mu}}\right)d^n x.
%\end{multline*}

On the other hand, let $(x^\mu,{A}_\mu^\alpha;{A}_{\mu,J}^\alpha)$ be adapted coordinates for $J^{2k-1}C(P)$. Observe that locally $C(P)\simeq T^*X\otimes \ad(P)$, so we will use these coordinates for both spaces indifferently. For the reduced section, we write $\sigma_{A}(x^\mu)=\overline\sigma_{A}(x^\mu)=(x^\mu,\sigma_\mu^\alpha(x^\mu))$ for some local functions $\sigma_\mu^\alpha\in C^\infty(X)$, $1\leq\mu\leq n$, $1\leq\alpha\leq m$. Since $\sigma_{A}=\left[j^1 s\right]_G$, Proposition \ref{prop:Thetakcoordinates} yields
\begin{equation*}
\sigma_\mu^\alpha=b_\beta^\alpha(s^\alpha)\partial_\mu s^\beta,\qquad 1\leq\mu\leq n,~1\leq\alpha\leq m.
\end{equation*}
From this and Lemma \ref{lemma:higherleibniz}, we get
\begin{align}\label{eq:proofnoether1}
& \frac{\partial}{\partial y_{I+J+1_\mu}^\alpha}\left(\frac{\partial^{|I+J+K|}}{\partial x^{I+J+K}}\left(b_\gamma^\beta(s^\alpha)\partial_\mu s^\gamma\right)\right)\\\nonumber
& \hspace{15mm}=\frac{\partial}{\partial y_{I+J+1_\mu}^\alpha}\left(\sum_{I^{(1)}+I^{(2)}=I+J+K}\frac{(I+J+K)!}{I^{(1)}!I^{(2)}!}\frac{\partial^{|I^{(1)}|}b_\gamma^\beta(s^\alpha)}{\partial x^{I^{(1)}}}\frac{\partial^{|I^{(2)}|}\partial_\mu s^\gamma}{\partial x^{I^{(2)}}}\right)\\\nonumber
& \hspace{15mm}=\frac{(I+J+K)!}{K!(I+J)!}\frac{\partial^{|K|} b_\alpha^\beta(s^\alpha)}{\partial x^K},
\end{align}
for each $1\leq\mu\leq n$, $1\leq\alpha,\beta\leq m$, $0\leq|I|\leq k-1$, $0\leq|J|\leq k-|I|-1$ and $0\leq|K|\leq k-|I|-|J|-1$. Moreover, by definition the extension of the reduced Lagrangian satisfies
\begin{equation*}
\hat l(x^\mu,\sigma_\mu^\alpha(x^\mu);\partial_J\sigma_\mu^\alpha(x^\mu))=\hat l\left(j_x^{k-1}\overline\sigma_{A}\right)=L\left(j_x^k s\right)=L(x^\mu,s^\alpha(x^\mu);\partial_J s^\alpha(x^\mu)),
\end{equation*}
where $\partial_J=\partial/\partial x^J$. By using this relation and equation \eqref{eq:proofnoether1}, we obtain
\begin{equation*}
\frac{\partial L}{\partial y_{I+J+1_\mu}^\alpha}\left(j^k s\right)=\sum_{|K|=0}^{k-|I|-|J|-1}\frac{(I+J+K)!}{K!(I+J)!}\frac{\partial\hat l}{\partial{A}_{\mu,I+J+K}^\beta}\left(j^{k-1}\sigma_{A}\right)\frac{\partial^{|K|} b_\alpha^\beta(s^\alpha)}{\partial x^K},
\end{equation*}
for each $1\leq\mu\leq n$, $0\leq|I|\leq k-1$ and $0\leq|J|\leq k-|I|-1$. For brevity, we have denoted ${A}_{\mu,J_0}^\alpha={A}_\mu^\alpha$ for $J_0=(0,\dots,0)$. By introducing this in \eqref{eq:proofnoether0} and by using Lemma \ref{lemma:higherleibniz}, we obtain
\begin{align}\label{eq:proofnoether2}
 \left(j^{2k-1}s\right)^*\mathcal J & =\sum_{|I|=0}^{k-1}\sum_{|J|=0}^{k-|I|-1}\sum_{|K|=0}^{k-|I|-|J|-1}(-1)^{|J|}\frac{(I+J+K)!}{K!(I+J)!}\xi^\beta\\\nonumber
& \hspace{10mm}\frac{\partial^{|J|}}{\partial x^J}\left(\frac{\partial\hat l}{\partial{A}_{\mu,I+J+K}^\beta}\left(j^{k-1}\overline\sigma_{A}\right)\frac{\partial^{|K|} b_\alpha^\beta(s^\alpha)}{\partial x^K}\right)\frac{\partial^{|I|}a_\beta^\alpha(s^\alpha)}{\partial x^I}d^{n-1}x_\mu\\\nonumber
& =\sum_{|I|=0}^{k-1}\sum_{|J|=0}^{k-|I|-1}\sum_{|K|=0}^{k-|I|-|J|-1}\sum_{|L|=0}^J (-1)^{|J|}\frac{(I+J+K)!J!}{K!(I+J)!L!(J-L)!}\xi^\beta\\\nonumber
& \hspace{10mm}\frac{\partial^{|L|}}{\partial x^L}\left(\frac{\partial\hat l}{\partial{A}_{\mu,I+J+K}^\beta}\left(j^{k-1}\sigma_{A}\right)\right)\frac{\partial^{|J+K-L|} b_\alpha^\beta(s^\alpha)}{\partial x^{J+K-L}}\frac{\partial^{|I|}a_\beta^\alpha(s^\alpha)}{\partial x^I}d^{n-1}x_\mu.
\end{align}
In addition, we have $b_\alpha^\beta(s^\alpha)a_\beta^\alpha(s^\alpha)=m$. Hence, from Lemma \eqref{lemma:higherleibniz} we get
\begin{equation*}
\frac{\partial^{|M|}}{\partial x^M}\left(b_\alpha^\beta(s^\alpha)a_\beta^\alpha(s^\alpha)\right)=\sum_{M^{(1)}+M^{(2)}=M}\frac{M!}{M^{(1)}!M^{(2)}!}\frac{\partial^{|M^{(1)}|}b_\alpha^\beta(s^\alpha)}{\partial x^{M^{(1)}}}\frac{\partial^{|M^{(2)}|}a_\beta^\gamma(s^\alpha)}{\partial x^{M^{(2)}}}=0,\qquad|M|>0.
\end{equation*}
By introducing this in \eqref{eq:proofnoether2}, we arrive at
\begin{equation*}
 \left(j^{2k-1}s\right)^*\mathcal J=m\sum_{|J|=0}^{k-1}(-1)^{|J|} \xi^\alpha\frac{\partial^{|J|}}{\partial x^J}\left(\frac{\partial\hat l}{\partial{A}_{\mu,J}^\alpha}\left(j^{k-1}\overline\sigma_{A}\right)\right)d^{n-1}x_\mu.
\end{equation*}
Thence, the Noether conservation law, ${\rm d}\left( \left(j^{2k-1}s\right)^*\mathcal J\right)=0$, locally reads
\begin{equation}\label{eq:noetherlocal}
\sum_{|J|=0}^{k-1}(-1)^{|J|}\frac{\partial^{|J+1_\mu|}}{\partial x^{J+1_\mu}}\left(\frac{\partial\hat l}{\partial {A}_{\mu,J}^\alpha}\left(j^{k-1} \sigma_{A}\right)\right)=0,\qquad 1\leq\alpha\leq m,
\end{equation}
where we have used that $\xi=\xi^\alpha B_\alpha\in\mathfrak g$ is arbitrary. 

Lastly, fix $x_0=(x_0^\mu)\in X$ and assume that our adapted coordinates are chosen so that ${A}$ is flat at $y=s(x_0)$, i.e., ${A}_{s(x_0)}=(dy^\alpha)_{s(x_0)}\otimes B_\alpha$. It is now clear that \eqref{eq:noetherlocal} are exactly the local equations computed in Corollary \ref{corollary:localequations} for $\overline\sigma_{A}=\hat 0$, i.e., $\sigma_\mu^\alpha=0$ for $1\leq\mu\leq n$ and $1\leq\alpha\leq m$.
\end{proof}

\section{Multivariate \texorpdfstring{$k$}{k}-splines on Lie groups}\label{sec:splines}

Variational splines are piecewise polynomial functions that allow for interpolating while minimizing some cost functional. For this reason, they have many applications in different areas such as optimal control theory, medical imaging or robotics (cf., for instance, \cite{La2009}). Here we focus on higher-order splines, i.e., those in which the cost functional depends on higher-order derivatives. Namely, the reduction theory for 1-dimensional $k$-splines on Lie groups developed in \cite[\S 3.2]{GaHoMeRaVi2012} (see also \cite{MaSiKr2010}) is extended to multivariate functions on Lie groups (cf. \cite{ChiOzMa2004}).

Let $G$ be a Lie group endowed with a right-invariant, (pseudo-)Riemannian metric, $\mathbf g$. The inner product induced on the Lie algebra, $\mathfrak g$, is denoted by the same symbol, $\mathbf g:\mathfrak g\times\mathfrak g\to\mathbb R$, and the induced norms on $TG$ and $\mathfrak g$ are denoted by $\left\|\cdot\right\|_\mathbf g:TG\to\mathbb R$ and $\left\|\cdot\right\|_{\mathfrak g}:\mathfrak g\to\mathbb R$, respectively. Similarly, the adjoint of $\ad_\xi$ induced by $\mathbf g$ is denoted by $\ad_\xi^{\normalfont\dagger}:\mathfrak g\to\mathfrak g$, i.e.,
\begin{equation*}
\mathbf g(\ad_\xi(\eta),\zeta)=\mathbf g(\eta,\ad_\xi^{\normalfont\dagger}(\zeta)),\qquad\xi,\eta,\zeta\in\mathfrak g.    
\end{equation*} 
On the other hand, let $X=\mathbb R^n$ (recall Remark \ref{remark:noncompact}) with the standard (global) coordinates and volume form given by $(x^\mu)=(x^1,\dots,x^n)$ and $d^n x=dx^1\wedge\dots\wedge dx^n$, respectively. The configuration bundle is given by
\begin{equation*}
\pi_{P,X}:P=X\times G\to X,\quad(x,g)\mapsto x.
\end{equation*}
The bundle being trivial yields $J^kP=J^k(X,G)$, the family of $k$-jets of functions $\varsigma:X\to G$, whose elements are denoted by $j_x^k\varsigma$. Let $\nabla:\mathfrak X(G)\times\mathfrak X(G)\to\mathfrak X(G)$ be the Levi--Civita connection of $\mathbf g$. Its pull-back by a (local) function $\varsigma:X\to G$, which is a linear connection on $\pi_{\varsigma^*TG,X}:\varsigma^*TG\to X$, is denoted by $\varsigma^*\nabla$ and, given $1\leq\mu\leq n$, the corresponding covariant derivative is denoted by
\begin{equation*}
\frac{\nabla\ }{\partial x^\mu}=\varsigma^*\nabla_{\partial_\mu}:\Gamma(\pi_{\varsigma^*TG,X})\to\Gamma(\pi_{\varsigma^*TG,X}).
\end{equation*}
%In particular, note that $\varsigma_\mu=\partial\varsigma/\partial x^\mu\in\Gamma(\pi_{\varsigma^*TG,X})$,

\begin{definition}
The Lagrangian for \emph{multivariate $k$-splines} on $G$ is given by
\begin{equation}\label{eq:lagrangianksplines}
L_s\left(j_x^k\varsigma\right)=\frac{1}{2}\kappa^\mu\left\|\left.\frac{\nabla^{k-1}\ }{\partial(x^\mu)^{k-1}}\right|_x(\partial_\mu\varsigma)\right\|_\mathbf g^2,\qquad j_x^k\varsigma\in J^k(X,G),
\end{equation}
for certain $\kappa^\mu\in\mathbb R$, $0\leq\mu\leq n$.
\end{definition}

Given a (local) function $\varsigma:X\to G$, we set
\begin{equation*}
\sigma=dR_{\varsigma^{-1}}\circ d\varsigma=\left(dR_{\varsigma^{-1}}\circ\partial_\mu\varsigma\right)dx^\mu\in\Gamma(\pi_{T^*X\otimes\mathfrak g,X})=\Omega^1(X,\mathfrak g),
\end{equation*}
where $R_g:G\to G$ denotes the right multiplication by $g\in G$. This induces the following isomorphism,
\begin{equation*}
J^1(X,G)\simeq G\ltimes(T^*X\otimes\mathfrak g),\quad j_x^1\varsigma\mapsto\left(\varsigma(x),\sigma(x)\right),
\end{equation*}
where $\ltimes$ denotes the fibered semidirect product given by the adjoint representation (cf. \cite[\S 4.3.1]{CaRo2023}). Hence, the isomorphism $\Xi_k$ given in \eqref{eq:diagram} becomes
\begin{equation*}
J^k(X,G)/G\ni j_x^k\varsigma\mapsto j_x^{k-1}\sigma\in C^{k-1}_0(P)
\end{equation*}
Observe that for the canonical principal connection on the trivial bundle $\pi_{P,X}$, the isomorphism \eqref{eq:isomorphismsCP} reduces to the identity and, thus, $\overline{C^{k-1}_0(P)}=C^{k-1}_0(P)$.

\begin{proposition}
The Lagrangian \eqref{eq:lagrangianksplines} is right invariant. Furthermore, a  (natural) extended reduced Lagrangian is
\begin{equation}\label{eq:reducedlagrangianksplines}
\hat\ell_s\left(j_x^{k-1}\sigma\right)=\frac{1}{2}\kappa^\mu\left\|\xi_\mu^{k-1}(x)\right\|_{\mathfrak g}^2,\qquad j_x^{k-1}\sigma\in J^{k-1}C(P),
\end{equation}
where $\sigma=\sigma_\mu dx^\mu$ and $\xi_\mu^j:X\to\mathfrak g$ are defined recursively as follows,
\begin{equation*}
\xi_\mu^0=\sigma_\mu,\qquad\xi_\mu^j=\partial_\mu\xi_\mu^{j-1}+\frac{1}{2}\left(\ad_{\sigma_\mu}^{\normalfont\dagger}\left(\xi_\mu^{j-1}\right)+\ad_{\xi_\mu^{j-1}}^{\normalfont\dagger}\left(\sigma_\mu\right)-\ad_{\sigma_\mu}\left(\xi_\mu^{j-1}\right)\right),~1\leq j\leq k-1.
\end{equation*}
\end{proposition}

\begin{proof}
Given a (local) function $\varsigma:X\to G$, suppose that 
\begin{equation}\label{proof:1}
\frac{\nabla^{k-1}}{\partial(x^\mu)^{k-1}}(\partial_\mu\varsigma)=dR_\varsigma\circ\xi_\mu^{k-1},\qquad1\leq\mu\leq n.
\end{equation}
In such case, the Lagrangian \eqref{eq:lagrangianksplines} may be written as
\begin{equation*}
L_s\left(j_x^k\varsigma\right)=\frac{1}{2}\kappa^\mu\left\|\left(dR_{\varsigma(x)}\right)_e\left(\xi_\mu^{k-1}(x)\right)\right\|_\mathbf g^2=\frac{1}{2}\kappa^\mu\left\|\xi_\mu^{k-1}(x)\right\|_{\mathfrak g}^2,
\end{equation*}
where right invariance of $\mathbf g$ has been used. It is now clear that $L$ is right invariant, since the maps $\xi_\mu^{k-1}$ do not depend on the value of $\varsigma$, but only on its derivatives (regarded as elements of the Lie algebra). By recalling that $\sigma=dR_{\varsigma^{-1}}\circ d\varsigma$, i.e., $\sigma_\mu=dR_{\varsigma^{-1}}\circ\partial_\mu\varsigma$, the reduced Lagrangian reads
\begin{equation*}
\ell_s\left(j_x^{k-1}\sigma\right)=L_s\left(j_x^k\varsigma\right)=\frac{1}{2}\kappa^\mu\left\|\xi_\mu^{k-1}(x)\right\|_{\mathfrak g}^2.
\end{equation*}
Although the reduced Lagrangian is computed for elements $j_x^{k-1}\sigma\in C^{k-1}_0(P)$, it can extended straightforwardly to the whole jet, $\hat\ell_s:J^{k-1}C(P)\to\mathbb R$, by taking arbitrary functions $\sigma_\mu:X\to\mathfrak g$, $1\leq\mu\leq n$.

Lastly, let us check \eqref{proof:1} by induction. For $k=1$ it is straightforward. Now let $k>1$. The formula for the Levi--Civita covariant derivative of the right-invariant metric $\mathbf g$ (cf. \cite[Eq. (3.13)]{GaHoMeRaVi2012} or \cite[\S 46.5]{KrMi1997}), together with the definition of pull-back connection, yield
\begin{equation*}
\varsigma^*\nabla_{\partial_\mu}\varsigma^*U=\varsigma^*(\nabla_{d\varsigma(\partial_\mu)}U)=dR_\varsigma\circ\left(d\phi_U\circ\varsigma_\mu+\frac{1}{2}\ad_{\sigma_\mu}^{\normalfont\dagger}(\phi_U)+\frac{1}{2}\ad_{\phi_U}^{\normalfont\dagger}(\sigma_\mu)-\frac{1}{2}\ad_{\sigma_\mu}(\phi_U)\right)\circ\varsigma
\end{equation*}
for each $U\in\mathfrak X(G)$, where $\phi_U\in C^\infty(G,\mathfrak g)$ is given by $\phi_U(g)=(dR_{g^{-1}})_g(U(g))$, $g\in G$. By using this and the induction hypothesis, we finish:
\begin{align*}
\frac{\nabla^{k-1}\ }{\partial(x^\mu)^{k-1}}(\partial_\mu\varsigma) & =\varsigma^*\nabla_{\partial_\mu}\left(\frac{\nabla^{k-2}\ }{\partial(x^\mu)^{k-2}}(\partial_\mu\varsigma)\right)=\varsigma^*\nabla_{\partial_\mu}\left(dR_\varsigma\circ\xi_\mu^{k-2}\right)\\
& =dR_\varsigma\circ\left(\partial_\mu\xi_\mu^{k-2}+\frac{1}{2}\ad_{\sigma_\mu}^{\normalfont\dagger}\left(\xi_\mu^{k-2}\right)+\frac{1}{2}\ad_{\xi_\mu^{k-2}}^{\normalfont\dagger}\left(\sigma_\mu\right)-\frac{1}{2}\ad_{\sigma_\mu}\left(\xi_\mu^{k-2}\right)\right)=dR_\varsigma\circ\xi_\mu^{k-1}.
\end{align*}
\end{proof}

Let us compute explicitly the reduced equations for $k=2$. This corresponds to the minimization of the accelerations, which in specific applications is directly related with the fuel expenditure of trajectories (cf. \cite{HuBl2004}). In such case, for each $\sigma=\sigma_\mu dx^\mu\in\Omega^1(X,\mathfrak g)$, it is easy to check that $\xi_\mu^1=\partial_\mu\sigma_\mu+\ad_{\sigma_\mu}^{\normalfont\dagger}\left(\sigma_\mu\right)$, $1\leq\mu\leq n$.

\begin{theorem}
The Euler--Poincaré field equations for \eqref{eq:reducedlagrangianksplines} with $k=2$, i.e.,
\begin{equation*}
\hat\ell_s\left(j_x^1\sigma\right)=\frac{1}{2}\kappa^\mu\left\|\partial_\mu\sigma_\mu+\ad_{\sigma_\mu}^{\normalfont\dagger}\left(\sigma_\mu\right)\right\|_{\mathfrak g}^2,\qquad j_x^1\sigma\in J^1C(P),
\end{equation*}
are given by
\begin{equation}\label{eq:EPksplines}
\kappa^\mu\left(\partial_\mu-\ad_{\sigma_\mu}^{\normalfont\dagger}\right)\left(\ad_{\eta_\mu }^{\normalfont\dagger}(\sigma_\mu)+\ad_{\eta_\mu }(\sigma_\mu)+\partial_\mu\eta_\mu \right)=0,
\end{equation}
where $\eta_\mu=\partial_\mu\sigma_\mu+\ad_{\sigma_\mu}^{\normalfont\dagger}(\sigma_\mu)$ for each $1\leq\mu\leq n$.
\end{theorem}

\begin{proof}
Let $\sigma=\sigma_\mu dx^\mu\in\Omega^1(X,\mathfrak g)$. The musical isomorphisms defined by $\mathbf g$ are denoted by $\sharp:\mathfrak g^*\to\mathfrak g$ and $\flat:\mathfrak g\to\mathfrak g^*$. For each $1\leq\mu,\nu\leq n$, it can be easily checked that 
\begin{equation}\label{proofksplines1}
\frac{\partial\hat\ell_s}{\partial(\partial_\nu\sigma_\mu)}\left(j^1\sigma\right)=\delta^{\mu\nu}\kappa^\mu\eta_\mu ^\flat,\qquad\partial_\nu\left(\frac{\partial\hat\ell_s}{\partial(\partial_\nu\sigma_\mu)}\left(j^1\sigma\right)\right)=\kappa^\mu\partial_\mu\eta_\mu ^\flat.
\end{equation}
As straightforward computation shows that $\ad_\xi^{\normalfont\dagger}(\zeta)=\ad_\xi^*(\zeta^\flat)^\sharp$ for each $\xi,\zeta\in\mathfrak g$. Thus, for each $1\leq\mu\leq n$ we have
\begin{align}\label{proofksplines2}
\frac{\partial\hat\ell_s}{\partial\sigma_\mu}\left(j^1\sigma\right) & =\kappa^\mu\mathbf g\left(\eta_\mu ,\ad^{\normalfont\dagger}(\sigma_\mu)+\ad_{\sigma_\mu}^{\normalfont\dagger}(\cdot)\right)\\\nonumber
& =\kappa^\mu\left(\mathbf g\left(\ad(\eta_\mu ),\sigma_\mu\right)+\mathbf g\left(\ad_{\sigma_\mu}(\eta_\mu ),\cdot\right)\right)=\kappa^\mu\left(-\mathbf g\left(\ad_{\eta_\mu }(\cdot),\sigma_\mu\right)-\mathbf g\left(\ad_{\eta_\mu}(\sigma_\mu ),\cdot\right)\right)\\\nonumber
& =-\kappa^\mu\mathbf g\left(\ad_{\eta_\mu }^{\normalfont\dagger}(\sigma_\mu)+\ad_{\eta_\mu}(\sigma_\mu ),\cdot\right)=-\kappa^\mu\left(\ad_{\eta_\mu }^*(\sigma_\mu^\flat)+\ad_{\eta_\mu }(\sigma_\mu)^\flat\right).
\end{align}
By gathering \eqref{proofksplines1} and \eqref{proofksplines2}, the Euler--Lagrange form of $\hat{\mathfrak l}_s=\hat\ell_s\,d^n x:J^1 C(P)\to\bigwedge^n T^*X$ is readily seen to be
\begin{align}\label{proofksplines3}
\left(j^2\sigma\right)^*\mathcal{EL}\left(\hat{\mathfrak l}_s\right) & =\left(\frac{\partial\hat\ell_s}{\partial\sigma_\mu}\left(j^1\sigma\right)-\partial_\nu\left(\frac{\partial\hat\ell_s}{\partial(\partial_\nu\sigma_\mu)}\left(j^1\sigma\right)\right)\right)d^n x\otimes\partial_\mu\\\nonumber
& =-\kappa^\mu\left(\ad_{\eta_\mu }^*(\sigma_\mu^\flat)+\ad_{\eta_\mu }(\sigma_\mu)^\flat+\partial_\mu\eta_\mu ^\flat\right)d^n x\otimes\partial_\mu\in\Omega^n(X,TX\otimes\mathfrak g^*).
\end{align}

On the other hand, since we have chosen ${A}_0$ as the canonical connection, we obtain the standard divergence, that is,
\begin{equation*}
\operatorname{div}^{{A}_0}:\Gamma(\pi_{TX\otimes\mathfrak g^*,X})\to C^\infty(X,\mathfrak g),\quad\eta_\alpha^\mu\,\partial_\mu\otimes B^\alpha\mapsto(\partial_\mu\eta_\alpha^\mu)B^\alpha,
\end{equation*}
where $\{B_\alpha\in\mathfrak g\mid1\leq\alpha\leq m\}$ is a basis of $\mathfrak g$ and $\{B^\alpha\in\mathfrak g^*\mid1\leq\alpha\leq m\}$ is its dual basis. In the same fashion, we have
\begin{equation*}
\ad_\sigma^*:\Gamma\left(\pi_{TX\otimes\mathfrak g^*,X}\right)\to C^\infty(X,\mathfrak g^*),\quad\eta_\alpha^\mu\partial_\mu\otimes B^\alpha\mapsto\eta_\alpha^\mu\,\ad_{\sigma_\mu}^*(B^\alpha).
\end{equation*}
From this and \eqref{proofksplines3}, we obtain the Euler--Poincaré field equations for $\hat\ell_s$,
\begin{align*}
0 & =\kappa^\mu\left(\operatorname{div}^{{A}_0}-\ad_\sigma^*\right)\left(\ad_{\eta_\mu }^*(\sigma_\mu^\flat)+\ad_{\eta_\mu }(\sigma_\mu)^\flat+\partial_\mu\eta_\mu ^\flat\right)d^n x\otimes\partial_\mu\\
& =\kappa^\mu\left(\partial_\mu-\ad_{\sigma_\mu}^*\right)\left(\ad_{\eta_\mu }^*(\sigma_\mu^\flat)+\ad_{\eta_\mu }(\sigma_\mu)^\flat+\partial_\mu\eta_\mu ^\flat\right)d^n x.
\end{align*}
By applying $\sharp$, we conclude.
\end{proof}

\begin{corollary}
If the metric $\mathbf g$ is bi-invariant, then the Euler--Poincaré field equations for $2$-splines \eqref{eq:EPksplines} become
\begin{equation*}
\kappa^\mu\left(\partial_\mu^3\sigma_\mu+\left[\sigma_\mu,\partial_\mu^2\sigma_\mu\right]\right)=0.
\end{equation*}
\end{corollary}

\begin{proof}
For a bi-invariant metric, we have $\ad_\xi^\dagger(\zeta)=-\ad_\xi(\zeta)$ for each $\xi,\eta\in\mathfrak g$. Therefore, $\eta_\mu=\partial_\mu\sigma_\mu$ for each $1\leq\mu\leq n$ and we finish.
\end{proof}

\begin{remark}[Elastica]
The incorporation of elastic terms to the Lagrangian for multivariate $k$-splines is motivated by optimal control problems where the velocities (first order derivatives) have to be minimized together with their accelerations \cite{NoHePa1989,CaSiCr1995,HuBl2004}. In such case, the Lagrangian is given by
\begin{equation}\label{eq:lagrangianelastica}
L_e\left(j_x^k\varsigma\right)=L_s\left(j_x^k\varsigma\right)+\frac{1}{2}\tau^\mu\left\|\partial_\mu\varsigma\right\|_{\mathbf g}^2,\qquad j_x^k\varsigma\in J^k(X,G),
\end{equation}
for certain $\tau^\mu\in\mathbb R$, $1\leq\mu\leq n$, with $L_s:J^k(X,G)\to\mathbb R$ as in \eqref{eq:lagrangianksplines}. The reduced Lagrangian, as well as the reduced equations, are computed by adding the extra term in the above results. Namely, it is clear that
\begin{equation*}
\hat\ell_e\left(j_x^{k-1}\sigma\right)=\hat\ell_s\left(j_x^{k-1}\sigma\right)+\frac{1}{2}\tau^\mu\left\|\sigma_\mu\right\|_{\mathfrak g}^2,\qquad j_x^{k-1}\sigma\in J^{k-1}C(P),
\end{equation*}
with $\hat\ell_s:J^{k-1}C(P)\to\mathbb R$ as in \eqref{eq:reducedlagrangianksplines}. The corresponding Euler--Poincaré field equations for $k=2$ are given by
\begin{equation*}\label{eq:EPelastica}
\kappa^\mu\left(\partial_\mu-\ad_{\sigma_\mu}^{\normalfont\dagger}\right)\left(\ad_{\eta_\mu }^{\normalfont\dagger}(\sigma_\mu)+\ad_{\eta_\mu }(\sigma_\mu)+\partial_\mu\eta_\mu \right)=\tau^\mu\left(\partial_\mu-\ad_{\sigma_\mu}^\dagger\right)\sigma_\mu.
\end{equation*}
Lastly, when the metric $\mathbf g$ is bi-invariant, the previous equations boil down to
\begin{equation}\label{eq:EPelasticabi}
\kappa^\mu\left(\partial_\mu^3\sigma_\mu+\left[\sigma_\mu,\partial_\mu^2\sigma_\mu\right]\right)=\tau^\mu\partial_\mu\sigma_\mu.
\end{equation}
\end{remark}

\subsection*{Funding}
MCL and ARA have been partially supported by Ministerio de Ciencia e Innovación (Spain) under grants PGC2018-098321-B-I00 and PID2021-126124NB-I00. ARA has been supported by Ministerio de Universidades (Spain) under an FPU grant.

%%%%%%%%%%%%%%%%%%%%%%%%%%%%%%%%%%%%%%%%%%%%%%%%%%
\bibliographystyle{plainnat}
\bibliography{biblio.bib}

\end{document}